\documentclass[a4paper,12pt]{amsart}
\usepackage[T1]{fontenc}    
\usepackage{enumerate}
\usepackage{amssymb}
\usepackage{tikz, subfigure}

\def\uu{\mathcal U}
\def\vv{\mathcal V}
\def\bb{\mathcal B}
\def\g{\mathcal G}

\newcommand{\set}[1]{\left\lbrace #1\right\rbrace}
\providecommand{\abs}[1]{\left\lvert#1\right\rvert}

\newcommand{\remove}[1]{ }
\newcommand{\qtq}[1]{\quad\text{#1}\quad}

\newtheorem{theorem}{Theorem}[section]
\newtheorem{proposition}[theorem]{Proposition}%
\newtheorem{lemma}[theorem]{Lemma}
\newtheorem{corollary}[theorem]{Corollary}

\theoremstyle{definition}

\theoremstyle{remark}
\newtheorem*{remark}{Remark}
\remove{}

\numberwithin{equation}{section}

\begin{document}

%
%
%
\title[Bases which admit exactly two expansions]{Bases which admit exactly  two expansions}
\author{Yi Cai*}
\address{College of Sciences, Shanghai Institute of Technology, Shanghai 201418,
People's Republic of China}
\email{546823540@qq.com}
\author{Wenxia Li}
\address{School of Mathematical Sciences, Shanghai Key Laboratory of PMMP, East China Normal University, Shanghai 200062,
People's Republic of China}
\email{wxli@math.ecnu.edu.cn}
\subjclass[2010]{Primary 11A63; Secondary 37F20, 37B10}
\keywords{$q$-expansion; unique $q$-expansion; quasi-greedy $q$-expansion;  generalized golden ratio.}
\thanks{*Corresponding author}

\begin{abstract}
For a positive integer $m$ let $\Omega _m=\{0,1, \cdots , m\}$ and
 \begin{align*}
 \bb_2(m)=&\left \{q\in(1,m+1]: \text{$\exists\; x\in [0, m/(q-1)]$ has exactly }\right. \\
&\left. \text{two different $q$-expansions w.r.t. $\Omega _m$}\right \}.
 \end{align*}
Sidorov \cite{S} firstly studied the set $\bb_2(1)$ and raised some questions. Komornik and Kong \cite{KK} further studied the set $\bb_2(1)$ and answered partial  Sidorov's questions. In the present paper, we consider the set $\bb_2(m)$ for general
positive integer $m$ and generalise the results obtained by Komornik and Kong.
\end{abstract}

\maketitle
\section{Introduction}
Given a positive integer $m$ and a real number $q\in(1,m+1]$, a $q$-\emph{expansion} with respect to the digit set $\Omega _m=\{0,1,\cdots , m\}$ of a real number $x\in I_q:=[0,\frac{m}{q-1}]$ is a sequence $(c_i)\in\Omega_m^\mathbb N$ such that
\begin{equation*}
x=\sum_{i=1}^{\infty}\frac{c_i}{q^i}:=(c_i)_q.
\end{equation*}

 In the past years the following set has attracted much more attention
 \begin{align*}
 \bb_k(1)=&\left \{q\in(1,2]: \text{$\exists\; x\in I_q$ has exactly }\right. \\
&\left. \text{$k$ different $q$-expansions w.r.t. $\Omega _1$}\right \}.
 \end{align*}
Erd\H os et al. \cite{P1,P2,P3} discovered that $\bb_k(1)\ne \emptyset$ for each $k\in\mathbb N\cup\set{\aleph_0}$.
For $k\geq 3$ very little was known for $\bb_k(1)$. However, much information on $\bb_2(1)$ has been discovered.
In 2009, Sidorov obtained

Theorem A \cite{S}
\begin{enumerate}[\upshape (i)]
\item $q\in\bb_2(1)\iff1\in\uu_q-\uu_q$;
\item $\uu\subseteq \bb_2(1)$;
\item $[T,2]\subset\bb_2(1)$, where $T\approx1.83929$ denotes the Tribonacci number, i.e., the positive zero of $q^3-q^2-q-1$;
\item The smallest two elements of $\bb_2(1)$ are $q_s\approx1.71064$, the root of
\begin{equation*}
q^4-2q^2-q-1=0
\end{equation*}
and $q_f\approx1.75488$ the root of
\begin{equation*}
q^3-2q^2+q-1=0.
\end{equation*}
\end{enumerate}
He  also posed some questions in \cite{S}, some of which was solved by Komornik and Kong \cite{KK}. 

For $q\in(1,m+1]$ we are interested in the sets
 \begin{align*}
 \bb_2(m)=&\left \{q\in(1,m+1]: \text{$\exists\; x\in I_q$ has exactly }\right. \\
&\left. \text{two different $q$-expansions w.r.t. $\Omega _m$}\right \},
 \end{align*}
and for $i\ge0$
\begin{equation*}
\mathcal B_2^{(i+1)}(m)=\set{q\in (1,m+1]: \text{$q$ is an accumulation point of $B_2^{(i)}(m)$}}
\end{equation*}
where $\mathcal B_2^{(0)}(m)=\mathcal B_2(m)$. The following results are from \cite[Theorems 1.3, 1.5. 1.7, 1.9]{KK}.

{\it
(I) The following conditions are equivalent:
(A) $q\in\bb_2(1)$;
(B) $1\in\uu_q-\uu_q$;
(C) $1\in\overline{\uu_q}-\overline{\uu_q}$;
(D) $1\in\vv_q-\vv_q$ \;\textrm{and}\; $q\neq\g$.

(II) $\overline{\uu}\subset\bb^{(\infty)}_2(1)$,
 $\vv\setminus\set{\g}\subset\bb^{(2)}_2(1)$.

(III) $\min \bb_2^{(1)}(1) = \min \bb_2^{(2)}(1) = q_f$, where $q_f$ is the largest real root of $q^3-2q+q-1=0$.

(IV) Every set $\bb_2^{(i)}(1)$ has infinitely many accumulation points in each connected
component $(q_0,q_0^*)$  of $(1,2] \setminus\uu$.

(V) $\bb_2(1)\cap(1, q_{KL})$ contains only algebraic integers, and hence it is countable.

(VI) $\bb_2(1)$ has infinitely many isolated points in $(1, q_{KL})$, and they are dense in $\bb_2(1)\cap(1, q_{KL})$.

(VII) If  $\vv\cap(1, q_{KL})=\set{q_n:n=1,2,\cdots}$, then
$q_{j+1}\le \min \bb_2^{(2j)}(1)<q_{2j+1}$ for all $j\ge0$,
and hence $\min\bb_2^{(j)}(1)\nearrow\min \bb_2^{(\infty)}(1) = q_{KL}$ as $j\to\infty$.

(VIII) For each $j = 0, 1,\cdots$, $\bb_2^{(j)}(1)\cap(1, q_{KL})$ has infinitely many isolated points, and they
are dense in $\bb_2^{(j)}(1)\cap(1, q_{KL})$.

(IX) For any $q\in\bb_2(1)$ we have
\begin{equation*}
\lim_{\delta\to 0}\dim_H(\bb_2(1)\cap(q-\delta, q+\delta))\le2\dim_H \uu_q.
\end{equation*}
}

In this paper we generalise the results in \cite{KK} for general alphabet $\Omega_m$ (for (V), (VI), (VIII) and (IX), one can check the proofs are independent of the alphabet), the main difficulty is in Section 3,  the proofs in Section 5 also become more complicated. We obtain the following results.

\begin{theorem}\label{T3}
The following conditions are equivalent:
\begin{enumerate}[\upshape (i)]
\item $q\in\bb_2(m)$;
\item $1\in\uu_q-\uu_q$;
\item $1\in\overline{\uu_q}-\overline{\uu_q}$;
\item $1\in\vv_q-\vv_q$, $q\neq\g(m)$.
\end{enumerate}
\end{theorem}

\begin{theorem}\label{T1}\mbox{}
\begin{enumerate}[\upshape (i)]
\item $\bb_2(m)$ is compact;
\item $\overline{\uu}\subset\bb^{(\infty)}_2(m)$;
\item $\vv\setminus\set{\g(m)}\subset\bb^{(2)}_2(m)$;
\item $\min \bb^{(1)}_2(m)=\min \bb^{(2)}_2(m)=q_f(m)$,  $q_f(m)$ is the largest real root of
\begin{equation*}
q^3-(k+2)q^2+q-k-1=0 \quad\text{if $m=2k+1$}
\end{equation*}
and
\begin{equation*}
q^2-(k+1)q-k=0 \quad\text{if $m=2k$}.
\end{equation*}
\end{enumerate}
\end{theorem}

The rest of the paper is arranged as follows. In Section 2 we recall some known results from unique expansions. Some results will be proved in section 3, which will be used for the proofs of main theorems. The proofs of main theorems are arranged in section 4. The final section  is devoted to the detailed description of unique expansions.

\section{Preliminaries}

In this section, we introduce some notation and list some important results. The \emph{greedy} $q$-expansion of $x\in I_q$ is the largest $q$-expansion in lexicographical order. The \emph{quasi-greedy} $q$-expansion of  $x\in I_q$  is the largest \emph{infinite} $q$-expansion in lexicographical order. A sequence is \emph{infinite} means that it does not have a last nonzero digit. Otherwise it is called \emph{finite}. 


We endow $\Omega_m^\mathbb N$ with the metric $D(\cdot,\cdot)$ defined as follows:
 \begin{align*}
 D((a_i),(b_i))=\left \{
 \begin{array}{cl}
 (m+1)^{-\min \{k: a_k\ne b_k\}}, & \text{if $(a_i)\neq(b_i)$} \\
 0, & \text{if $(a_i)=(b_i)$. }
 \end{array}
 \right.
 \end{align*}

Many works were devoted to the following sets
\begin{equation}\label{uqm}
\uu_q=\set{x\in I_q:\text{$x$ has a unique $q$-expansion w.r.t. $\Omega _m$}},
\end{equation}
and
\begin{equation*}
\vv_q=\set{x\in I_q:\text{$x$ has at most one doubly infinite $q$-expansion}}.
\end{equation*}
We call a $q$-expansion $(a_i)$ w.r.t. $\Omega _m$ is \emph{doubly infinite} if both $(a_i)$ and its \emph{reflection} $\overline{(a_i)}=(m-a_i)$ are infinite. The following two subsets of $(1, m+1]$ are very important
\begin{align*}
&\uu=\set{q\in(1,m+1]:1\in\uu_q},\\
&\vv=\set{q\in(1,m+1]:1\in\vv_q}.
\end{align*}
Clearly we have $\uu_q\subseteq \vv_q$ and $\uu\subseteq \vv$.

When $m=1$, Glendinning and Sidorov \cite{SN2} showed that $\uu_q$ is countable for $q\in(\g,q_{KL})$  and uncountable with positive Hausdorff
dimension for $q\in(q_{KL},2]$, where $\g =(1+\sqrt{5})/2$ is the \emph{golden radio}, $q_{KL}$ know as the \emph{Komornik-Loreti} constant (see \cite{KL0}).
Komornik and Loreti \cite{KL1}, de Vries, Komornik and Loreti \cite{VKL} proved that $\uu$ is closed from above but not from below, and its closure $\overline{\uu}$ is a Cantor set. Komornik and Loreti \cite{KL0,KL2} proved that  $\uu$ has a smallest element $q_{KL}$ known as the Komornik-Loreti constant. Kong and Li \cite{KL}, Komornik, Kong and Li \cite{KKL} determined the Hausdorff dimensions of $\uu_q$.
Furthermore $(1,m+1]\setminus\overline{\uu}=\cup(p_0,p_0^*)$, where $p_0$ runs over $\set{1}\cup(\overline{\uu}\setminus\uu)$ and $p_0^*$ runs over a proper subset of $\overline{\uu}$, and $\vv\cap(p_0,p_0^*)=\set{q_\ell:\ell=1,2\cdots}$ is a strictly increasing sequence converging to $p_0^*$. Especially, $\vv\cap(1,q_{KL})$ has a smallest element $\g(m)$ known as the \emph{generalized golden ratio}, which was given by Baker \cite{B}. More exactly,
\begin{equation}\label{Gm}
\begin{split}
\mathcal G(m)
=\left \{
\begin{array}{ll}
 k+1, & \text{if}\; m=2k, k=1,2,\cdots   \\
 \frac{k+1+\sqrt{k^{2}+6k+5}}{2}, & \text{if}\; m=2k+1, k=0,1,\cdots  \\
\end{array}
\right.
\end{split}
\end{equation}
and
\begin{equation}\label{KL}
\begin{split}
\alpha(q_{KL})
=\left \{
\begin{array}{ll}
 (k+\tau_i-\tau_{i-1}), & \text{if}\; m=2k, k=1,2,\cdots   \\
 (k+\tau_i), &\text{if}\; m=2k+1, k=0,1,\cdots  \\
\end{array}
\right.
\end{split}
\end{equation}
where $\alpha(q_{KL})$ is the quasi-greedy $q$-expansion of 1 and $(\tau_i)$ denotes the classical Thue-Morse sequence.

This generalized golden ratio $\g(m)$ plays an important part in the sense that  for $q \in (1,\g(m))$ every $x \in (0,\frac{m}{q-1})$ has uncountable $q$-expansions, and for $q \in (\g(m),m+1]$ there exists $x \in (0,\frac{m}{q-1})$ that has a unique $q$-expansion. Thus   $\uu_q=\set{0,\frac{m}{q-1}}$ for all $q\in(1,\g(m)]$ (see, e.g., \cite{P3,SV}).
In the whole paper, we always use $\alpha (q)=(\alpha_i)$ and $\beta (q)=(\beta_i)$ to denote the quasi-greedy and greedy $q$-expansions of $1$ respectively. The following property of quasi-greedy expansion of $1$ is related to Parry's work \cite{P} (see also \cite{BK,DK}).
\begin{lemma}\label{16}
The map $q\mapsto \alpha(q)$ is a strictly increasing bijection from $(1,m+1]$ onto the set of all infinite sequences $(\alpha _i)$ satisfying the inequality
\begin{equation*}
\alpha_{n+1}\alpha_{n+2}\cdots\le \alpha_1\alpha_2\cdots  \quad\text{for all $n\ge 0$}.
\end{equation*}
Moreover, the map $q\mapsto \alpha(q)$ is continuous from the left.
\end{lemma}

Similarly, combing Lemma 2.1 and Proposition 2.5 of \cite{BK} we have

\begin{lemma}\label{161}
The map $q\mapsto \beta(q)$ is a strictly increasing bijection from $(1,m+1)$ onto the set of all sequences $(\beta _i)$ satisfying the inequality
\begin{equation*}
\beta_{n+1}\beta_{n+2}\cdots <\beta_1\beta_2\cdots  \quad\text{for all $n\ge1$},
\end{equation*}
Moreover, the map $q\mapsto \beta(q)$ is continuous from the right.
\end{lemma}
We remark that $\beta(m+1)=m^\infty $.
Let $\uu'_q$ be the set of
the corresponding $q$-expansions of all elements in $\uu_q$ defined by (\ref{uqm}). We recall the following characterization of unique expansion \cite{BK}.
\begin{lemma}\label{L22}
Let $q\in(1,m+1]$, then $(c_i)\in\uu'_q$ if and only if
\begin{equation}\label{uicodedes}
\begin{split}
&c_{n+1}c_{n+2}\cdots<\alpha_1(q)\alpha_2(q)\cdots  \quad\text{whenever $c_n<m$} \\
&\overline{c_{n+1}c_{n+2}\cdots}<\alpha_1(q)\alpha_2(q)\cdots  \quad\text{whenever $c_n>0$ }.
\end{split}
\end{equation}
\end{lemma}
In fact, it is easy to check that  (\ref{uicodedes}) is equivalent to
\begin{equation}\label{uicodedes1}
\begin{split}
&c_{k+1}c_{k+2}\cdots<\alpha_1(q)\alpha_2(q)\cdots ,\;\textrm{when}\; c_1\cdots c_k\ne m^k \\
&\overline{c_{k+1}c_{k+2}\cdots}<\alpha_1(q)\alpha_2(q)\cdots ,\; \textrm{when}\; c_1\cdots c_k\ne 0^k.
\end{split}
\end{equation}

Komornik and Loreti in \cite{VKL} provided useful characterizations of $\mathcal U, \overline{\mathcal U}$ and $ \mathcal V$ (see
\cite[Theorem 2.5, 3.9, Lemma 3.15]{VKL}.
\begin{lemma}\label{17}\mbox{}
\begin{enumerate}[\upshape (i)]
\item $q\in\uu\setminus\set{m+1}$ if and only if $\alpha(q)=(\alpha_i(q))$ satisfies
\begin{equation*}
\overline{\alpha(q)}<\alpha_{n+1}(q)\alpha_{n+2}(q)\cdots< \alpha(q)  \quad\text{for all $n\ge1$}.
\end{equation*}
\item $q\in\overline{\uu}$ if and only if $\alpha(q)=(\alpha_i(q))$ satisfies
\begin{equation*}
\overline{\alpha(q)}<\alpha_{n+1}(q)\alpha_{n+2}(q)\cdots\le \alpha(q)  \quad\text{for all $n\ge1$}.
\end{equation*}
\item $q\in\vv$ if and only if $\alpha(q)=(\alpha_i(q))$ satisfies
\begin{equation}\label{15}
\overline{\alpha(q)}\le\alpha_{n+1}(q)\alpha_{n+2}(q)\cdots\le \alpha(q)  \quad\text{for all $n\ge1$}.
\end{equation}
\end{enumerate}
\end{lemma}
Hence, $\mathcal U\subseteq  \overline{\mathcal U}\subseteq  \mathcal V$. The following lemma comes from \cite[Theorem 1.2, Lemmas 3.11, 3.14]{VKL}. For a finite word $a_1\cdots a_{n-1}a_n\in \Omega ^n_m$ define
\begin{align*}
&a_1\cdots a_{n-1}a_n^+=a_1\cdots a_{n-1}(a_n+1)\qtq{if} a_n<m,\\
&a_1\cdots a_{n-1}a_n^-=a_1\cdots a_{n-1}(a_n-1)\qtq{if} a_n>0.
\end{align*}
\begin{lemma}\label{18}\mbox{}
\begin{enumerate}[\upshape (i)]
\item Every $q\in\overline{\uu}\setminus\uu$ there exists a sequence $(q_n)\in\uu$ satisfying $(q_n)\nearrow q$ as $n\rightarrow\infty$.
\item Every $q\in\overline{\uu}\setminus\uu$ the quasi-greedy expansion $\alpha(q)$ is periodic.
\item Every $q\in\vv\setminus(\overline{\uu}\cup\set{\g(m)})$ there exists a word $a_1\cdots a_n$ with $n\ge1$ such that
\begin{equation*}
\alpha(q)=(a_1\cdots a_{n-1}a_n^+\overline{a_1\cdots a_{n-1} a_n^+})^\infty
\end{equation*}
where $(a_1\cdots a_n)^\infty$ satisfies \eqref{15}.
\end{enumerate}
\end{lemma}
The following lemma is from \cite[Theorems 1.4, 1.5]{VK} which is helpful in our proof.
\begin{lemma}\label{19}\mbox{}
\begin{enumerate}[\upshape (i)]
\item $\uu_q$ is closed if and only if $q\in(1,m+1]\setminus\overline{\uu}$.
\item $\uu_q=\overline{\uu_q}=\vv_q$ if and only if $q\in(1,m+1]\setminus\vv$.
\end{enumerate}
\end{lemma}

\section{$B_{2}(m)$ with multiple alphabet}
 In this section we give a description of $\mathcal B_{2}(m)$. We first establish an  elementary lemma.
\begin{lemma}\label{L1}
Let $q\in(1,m+1]$. If $x\in  I_q=[0,\frac{m}{q-1}]$ has two different $q$-expansions $(a_{i})$ and $(b_{i})$ w.r.t. $\Omega_{m}$ satisfying
\begin{equation*}
a_{1}\cdots a_{k-1}=b_{1}\cdots b_{k-1}\qtq{and}b_{k}=a_{k}+d, d\geq2.
\end{equation*}
Then $x$ has the $q$-expansions starting with $a_{1} \cdots a _{k-1}(a_{k}+n)$, for all $1\le n \le d-1$. 
\end{lemma}
\begin{proof}
Indeed, we have
\begin{equation*}
(a_{1} \cdots a _{k-1}(a_{k}+n)0^{\infty})_{q}<(b_{1}b _{2}\cdots)_q=x
\end{equation*}
and
\begin{equation*}
(a_{1} \cdots a _{k-1}(a_{k}+n)m^{\infty})_{q}>(a_{1}a _{2}\cdots)_q=x.
\end{equation*}
Hence $q^{k}(x-(a_{1} \cdots a _{k-1}(a_{k}+n)0^{\infty})_{q}) \in [0,\frac{m}{q-1}]$, and there exists a sequence $c_{k+1}c_{k+2}\cdots \in \Omega^{\mathbb N}_{m}$ such that
$q^{k}(x-(a_{1} \cdots a _{k-1}(a_{k}+n)0^{\infty})_{q})=(c_{k+1}c_{k+2}\cdots)_q,$ which is equivalent to $(a_{1} \cdots a _{k-1}(a_{k}+n)c_{k+1}c_{k+2}\cdots)_{q}\\
=x$.
\end{proof}
 We deduce from the preceding lemma a corollary as follows.
\begin{corollary}\label{L2}
Let $q\in(1,m+1]$. If $x$ has exactly two different $q$-expansions $(a_{i})$ and $(b_{i})$ w.r.t. $\Omega_{m}$ satisfying
\begin{equation*}
a_{1} \cdots a _{k-1}=b_{1} \cdots b _{k-1}\;\;\; and\; \;\; b_{k} > a_{k},
\end{equation*}
then $b_{k}=a_{k}+1.$
\end{corollary}

\begin{theorem}\label{T4}
For $q\in(1,m+1]$,  $q\in \bb_{2}(m)$ if and only if $1\in\uu_{q}-\uu_{q}.$
\end{theorem}
\begin{proof}
It suffices to take $q\in (\mathcal G(m),m+1]$, because $\uu_{q}=\left \{0, \frac{m}{q-1}\right \}$ if $1<q\leq \mathcal G(m)$.

If $q\in\mathcal B_{2}(m),$ then there exists $x \in (0,\frac{m}{q-1})$ having exactly two $q$-expansions $(a_{i})$ and $(b_{i})$ w.r.t.$\Omega _m$.  Suppose that $a_{1}\cdots a _{k-1}=b_{1}\cdots b _{k-1}$ and $b_{k}=a_{k}+1$ for some $k\geq1$ by Corollary \ref{L2}. Thus, the equality $x=(a_i)_q=(b_i)_q$ implies
$$
(0a_{k+1}a_{k+2}\cdots )_q=(1b_{k+1}b_{k+2}\cdots )_q
$$
and so
$$
1=(a_{k+1}a_{k+2}\cdots )_q-(b_{k+1}b_{k+2}\cdots )_q.
$$
By the assumption we have $(a_{k+1}a_{k+2}\cdots )_q, (b_{k+1}b_{k+2}\cdots )_q\in \uu_{q}$.


 Conversely, if $1 \in \uu_{q}-\uu_{q},$ then there exist $(c_{i})$ and $(d_{i})\in\uu_{q}'$ such that
 $$
 1+(c_{i})_{q}=(d_{i})_{q}, \;i.e., \;x:=(0d_{1}d_{2}\cdots)_{q}=(1c_{1}c_{2}\cdots)_{q}.
  $$
 Thus $x$ has no more $q$-expansions w.r.t. $\Omega _m$ starting with $0$ or $1$. On the other hand,
 we claim that any $q$-expansion w.r.t. $\Omega _m$ of $x$ can not  start with $2\leq c\leq  m$. Otherwise
 $$
 \frac{c}{q}\le x=(0d_{1}d_{2}\cdots)_{q}\le \frac{m}{q(q-1)}
 $$
 which leads to
 $2\le c\le\frac{m}{q-1}$. Hence  $q\le 1+\frac{m}{2}\le \mathcal G(m)$ by (\ref{Gm}). However, $\uu_q=\set{0,\frac{m}{q-1}}$ for all $q\in (1,\g(m)]$, which leads to $1\notin \uu_{q}-\uu_{q}$.
\end{proof}
Now we give a equivalent characterization of Theorem \ref{T4}.
\begin{lemma}\label{3}
For $q \in (1,m+1]$, $q\in \mathcal B_{2}(m)$ if and only if there exist two sequences $(c_{i}), (d_{i}) \in \mathcal U^{\prime}_{q}$ satisfying the equality
\begin{equation*}
((n+1)(c_{i}))_{q}=(n(d_{i}))_{q}
\end{equation*}
for all $n=0,1, \cdots , m-1$.
\end{lemma}
 For $q \in (1,m+1]$ we set
\begin{equation*}
 A'_{q}:=\set{(c_{i})\in \mathcal U^{\prime}_{q}:0\le c_{1}<\alpha_1(q)}.
\end{equation*}
 By the definition of $A'_{q}$ every sequence $(c_{i}) \in A'_{q}$ satisfies (cf. \cite{BK})
\begin{equation*}
c_{i+1}c_{i+2}\cdots<\alpha(q)
\end{equation*}
for all $i\geq0$ by (\ref{uicodedes1}). Hence
\begin{equation}\label{e30}
\mathcal U^{\prime}_{q}=\bigcup_{\textnormal{c}\in A^{\prime}_{q}}\{\textnormal{c},\overline{\textnormal{c}}\}.
\end{equation}
 Indeed, this holds for $q\in (1, \mathcal G(m)]$ by
 $\mathcal U_{q}=\{0, m/(q-1)\}$.
 Suppose that $q\in (\mathcal G(m), m+1]$.  Let $(d_i)\in \mathcal U^{\prime}_{q}$ with $d_1\geq \alpha _1(q)$. Then $(\overline{d_i})\in \mathcal U^{\prime}_{q}$. We claim that $\overline{d_1}<\alpha _1(q)$ by the following reasoning:

Case 1. $m=2k$. Since $q>\mathcal G(m)=k+1$ we have $\alpha _1(q)\ge k+1$. Thus $2\alpha _1(q)>m$.

Case 2. $m=2k+1$. Since $q>\mathcal G(m)>k+1$ we have $\alpha _1(q)\ge k+1$. Thus $2\alpha _1(q)>m$.

In both cases, we have $m-d_1\leq m-\alpha _1(q)<\alpha _1(q)$. We remark that it is possible that both $d_1< \alpha _1(q)$ and $\overline{d_1}< \alpha _1(q)$.

%

\begin{lemma}\label{5}
For $q \in (1,m+1]$, $q\in \bb_{2}(m)$ if and only if $q$ is a zero of the function
\begin{equation}\label{e33}
f_{\textnormal{c},\textnormal{d}}(t)=(1\textnormal{c})_{t}+(m\textnormal{d})_{t}-(m^{\infty})_{t}
\end{equation}
for some $\textnormal{c},\textnormal{d} \in A'_{q}$, i.e., $(1\textnormal{c})_q+(m\textnormal{d})_q=(m^{\infty})_q$.
\end{lemma}
\begin{proof}
It follows from \eqref{e30} and Lemma \ref{3} that $q \in \mathcal B_{2}(m)$ if and only if $q$ satisfies one of the following equations for some $\textnormal{c},\textnormal{d} \in A^{\prime}_{q}:$
\begin{equation}\label{li0}
(1\textnormal{c})_{q}=(0\textnormal{d})_{q}, (1\textnormal{c})_{q}=(0\overline{\textnormal{d}})_{q}, (1\overline{\textnormal{c}})_{q}=(0\textnormal{d})_{q}\; \textrm{and}\; (1\overline{\textnormal{c}})_{q}=(0\overline{\textnormal{d}})_{q}.
\end{equation}
We claim that $q$ only satisfies the second equation.
For $\textnormal{d} \in A'_{q}$  one has that $(\textnormal{d})_{q}<\frac{\alpha _1(q)}{q}$. Thus for any $s\in \{0,1,\cdots ,m\}^{\mathbb N}$
\begin{equation}\label{li1}
(0\textnormal{d})_{q}=\frac{1}{q}(\textnormal{d})_{q}<\frac{\alpha _1(q)}{q^2}\leq \frac{1}{q}=(10^{\infty})_{q}\le (1\textnormal{s})_{q}.
\end{equation}
Hence for any $\textnormal{c},\textnormal{d} \in A^{\prime}_{q}$ one has
\begin{equation*}
(1\textnormal{c})_{q}>(0\textnormal{d})_{q}\;\textrm{and}\; (1\overline{\textnormal{c}})_{q}>(0\textnormal{d})_{q}.
\end {equation*}
Finally the fourth equality $(1\overline{\textnormal{c}})_{q}=(0\overline{\textnormal{d}})_{q}$ in (\ref{li0}) is equivalent to $((m-1)\textnormal{c})_q=(m\textnormal{d})_{q}$, and then is equivalent to
$(0\textnormal{c})_q=(1\textnormal{d})_{q}$. However, by (\ref{li1}) one has that $(0\textnormal{c})_q<(1\textnormal{d})_{q}$.

%
%

We complete proof by the equality $(1\textnormal{c})_{q}-(0\overline{\textnormal{d}})_{q}=(1\textnormal{c})_{q}+(m\textnormal{d})_{q}-(m^{\infty})_{q}.$
\end{proof}
We rewrite \eqref{e33} as
\begin{equation}\label{e34}
f_{\textnormal{c},\textnormal{d}}(t)=((m+1)(c_i+d_i))_{t}-(m^{\infty})_{t},
\end{equation}
where $\textnormal{c}=(c_i)$ and $\textnormal{d}=(d_i)$.
 It is natural to observe  the following properties.
\begin{lemma}\label{6}
Let $q\in(1,m+1]$ and $\textnormal{c},\textnormal{d} \in A'_{q}.$
\begin{enumerate}[\upshape (1)]
\item $f_{\textnormal{c},\textnormal{d}}(t)$ is symmetric w.r.t $(\textnormal{c},\textnormal{d})$, i.e., $f_{\textnormal{c},\textnormal{d}}(t)=f_{\textnormal{d},\textnormal{c}}(t)$;
\item $f_{\textnormal{c},\textnormal{d}}(t)\in C([q, m+1])$
 and $f_{\textnormal{c},\textnormal{d}}(q)$ is continuous w.r.t. $(\textnormal{c},\textnormal{d})\in  A'_{q}\times  A'_{q}$;
\item If $\textnormal{c}' \in A^{\prime}_{q}$ and $\textnormal{c}'>\textnormal{c}$, then $f_{\textnormal{c}',\textnormal{d}}(p)>f_{\textnormal{c},\textnormal{d}}(p)$ for all $p\ge q$. Similarly, if $\textnormal{d}' \in A'_{q}$ and $\textnormal{d}'>\textnormal{d}$, then $f_{\textnormal{c},\textnormal{d}'}(p)>f_{\textnormal{c},\textnormal{d}}(p)$ for all $p\ge q$;
\item $f_{\textnormal{c},\textnormal{d}}(m+1)\geq 0$.
\end{enumerate}
\end{lemma}
\begin{proof}
(1) It just follows from the definition \eqref{e34} of $f_{\textnormal{c},\textnormal{d}}(q)$.


(2) Firstly we point out that for given $\textnormal{c},\textnormal{d}\in A'_{q}$, $f_{\textnormal{c},\textnormal{d}}(t)$ is well-defined for $t\in [q, m+1]$ because $\textnormal{c},\textnormal{d}\in A'_{q}\subseteq A'_{t}$. Note that
\begin{align*}
f_{\textnormal{c},\textnormal{d}}(t)&=((m+1)(c_i+d_i))_{t}-(m^{\infty})_{t}\\
&=\frac{m+1}{t}-\frac{m}{t-1}+\sum _{k=2}^\infty \frac{c_{k-1}+d_{k-1}}{t^k}.
\end{align*}
Denote $S(t)=\sum _{k=2}^\infty \frac{c_{k-1}+d_{k-1}}{t^k}$. We show $S(t)$ is continuous in $[q, m+1]$.  Arbitrarily take $[\alpha , \beta ]\subseteq [q,m+1]$. Note that
$$
\frac{c_{k-1}+d_{k-1}}{t^k}\leq \frac{2m}{\alpha ^k}, \forall t\in [\alpha , \beta ]\;\textrm{and}\; \sum _{k=2}^\infty \frac{2m}{\alpha ^k}<+\infty .
$$
Thus $\sum _{k=2}^\infty \frac{c_{k-1}+d_{k-1}}{t^k}$ converges uniformly in $[\alpha , \beta ]$. So $S(t)\in C([\alpha , \beta ])$.

Now let $\textnormal{c}_n=(c_{n,i}), \textnormal{d}_n=(d_{n.i})\in A'_q$ be such that $\textnormal{c}_n\to \textnormal{c} $ and $\textnormal{d}_n\to \textnormal{d} $. Then
for any $k\in \mathbb N$ there exist $\ell =\ell (k)\in \mathbb N$ such that $c_{n,1}c_{n,2}\cdots c_{n,k}=c_1c_2\cdots c_k$
and $d_{n,1}d_{n,2}\cdots d_{n,k}=d_1d_2\cdots d_k$ whenever $n\geq \ell $. Then for $n\geq \ell $ we have
\begin{align*}
&\abs{f_{\textnormal{c}_n,\textnormal{d}_n}(q)-f_{\textnormal{c},\textnormal{d}}(q)}
\le \frac{m}{q^k(q-1)}+\frac{m}{q^k(q-1)}=\frac{2m}{q^k(q-1)}.
\end{align*}

(3) Note that $\textnormal{c}' , \textnormal{c}\in A^{\prime}_{q}$ and $\textnormal{c}'>\textnormal{c}$ implies $(\textnormal{c}')_q>(\textnormal{c})_q$. Since $\uu'_q\subset\uu'_p$ for all $p\ge q$, we have $A'_q\subset A'_p$. The desired result just follows from \eqref{e34}.

(4) We have
\begin{align*}
f_{\textnormal{c},\textnormal{d}}(m+1)&=((m+1)(c_i+d_i))_{m+1}-(m^\infty)_{m+1}\\
&=(0(c_i+d_i))_{m+1}\geq 0,
\end{align*}
as desired.
\end{proof}

We recall that $q_f(m)$ is the largest real root of
\begin{equation}\label{qfeven}
q^2-(k+1)q-k=0 \quad\text{when $m=2k$}
\end{equation}
and
\begin{equation}\label{qfodd}
q^3-(k+2)q^2+q-k-1=0 \quad\text{when $m=2k+1$}.
\end{equation}
We emphasize the important property:
$$
k+2>q_f(2k)=\frac{k+1+\sqrt{k^2+6k+1}}{2}>\mathcal G(2k)=k+1
$$
and
$$
k+2>q_f(2k+1)>\mathcal G(2k+1)=\frac{k+1+\sqrt{k^{2}+6k+5}}{2}>k+1.
$$
Since $\mathcal B_{2}(m)\cap(1,q_f(m))$ is a finite discrete set (see \cite[Propositions 3.5 and 4.9]{KLZ}),  we shall focus to $B_{2}(m)\cap[q_f(m),m+1]$.   For an infinite sequence $\alpha =\alpha _1\alpha _2\cdots $, write $\alpha |_n=\alpha _1\cdots \alpha _n$ and $\alpha |_{s,n}=\alpha _s\cdots \alpha _n$ for $s\leq n$. A simple fact will frequently occur in the following lemmas. We list it as a proposition without proof.
%
%
%
%
%
\begin{proposition}\label{simlemma}
Let $h(x)\in C^3([a,b])$. We have  $h(x)>0$ for $x\in [a,b]$ if the following conditions hold:

(I)   $h'''(x)\geq 0, x\in [a,b]$, or $h'''(x)\leq 0, x\in [a,b]$, or there exist $a<c <b$ such that $h'''(x)\geq 0, x\in [a, c]$ and $h'''(x)\leq 0, x\in [c,b]$;

(II)   $h''(a)>0$, $h'(a)>0$, $h(a)>0$ and $h(b)>0$.
\end{proposition}





%
%
%
We first consider the case of $m$ being odd.
\begin{lemma}\label{9}
 Let $m=2k+1$, $q \in [q_f(m),m+1]$ and $\textnormal{c},\textnormal{d} \in A^{\prime}_{q}$ w.r.t. $\Omega_{m}$.
\begin{enumerate}[\upshape (1)]
\item  If $k=0$ and $\textnormal{c}+\textnormal{d}\ge 00120^\infty$, then $f_{\textnormal{c},\textnormal{d}}(q)> 0$.
\item  Let $k=1$. If $(\textnormal{c}+\textnormal{d})|_1=0$ and $\textnormal{c}+\textnormal{d} \ge 0450^\infty$, then $f_{\textnormal{c},\textnormal{d}}(q)> 0$.  If $\textnormal{c}+\textnormal{d} \ge 120^\infty$, then $f_{\textnormal{c},\textnormal{d}}(q)> 0$.
\item Let $k\ge2$. If $(\textnormal{c}+\textnormal{d})|_1=k-1$ and $\textnormal{c}+\textnormal{d}\ge (k-1)(m+2)0^\infty$, then $f_{\textnormal{c},\textnormal{d}}(q)>0$. If $\textnormal{c}+\textnormal{d} \ge k(k+1)0^\infty$, then $f_{\textnormal{c},\textnormal{d}}(q)>0$.
\item If $k=0$ and $\textnormal{c}+\textnormal{d} <00120^\infty$, then $f_{\textnormal{c},\textnormal{d}}(t)$ is strictly increasing for $t\in [q,2]$.
\item  Let $k=1$. If $\textnormal{c}+\textnormal{d} <0450^\infty$, then $f_{\textnormal{c},\textnormal{d}}(t)$ is strictly increasing for $t\in [q,4]$. If $(\textnormal{c}+\textnormal{d})|_1=1$ and $\textnormal{c}+\textnormal{d}<120^\infty$, then $f_{\textnormal{c},\textnormal{d}}(t)$ is strictly increasing for $t\in [q,4]$.
\item Let $k\ge2$. If $\textnormal{c}+\textnormal{d}<(k-1)(m+2)0^\infty$, then $f_{\textnormal{c},\textnormal{d}}(t)$  is strictly increasing for $t\in [q,m+1]$. If $(\textnormal{c}+\textnormal{d})|_1=k$ and $\textnormal{c}+\textnormal{d}<k(k+1)0^\infty$, then $f_{\textnormal{c},\textnormal{d}}(t)$  is strictly increasing for $t\in [q,m+1]$.
\end{enumerate}
\end{lemma}

\begin{proof}

(1)  Since $q_f(1)\le q\le2$, we have
\begin{align*}
\min _{\alpha \ge 00120^\infty }(\alpha )_q=\min \{(00120^\infty )_q, (010^\infty )_q, (0020^\infty )_q\}=(010^\infty )_q.
\end{align*}
Thus when  $(\textnormal{c}+\textnormal{d}) \ge 00120^\infty$, for  $q\in [q_f(1), 2]$ we have
\begin{align*}
f_{\textnormal{c},\textnormal{d}}(q)&=((1+1)(c_i+d_i))_{q}-(1^{\infty})_{q}=\frac{2}{q}+\frac{1}{q}(\textnormal{c}+\textnormal{d}) _q-\frac{1}{q-1}\\
&> \frac{2}{q}+\frac{1}{q}(010^\infty )_q-\frac{1}{q-1}
=\frac{2}{q}+\frac{1}{q^3}-\frac{1}{q-1}\\
&=\frac{q^3-2q^2+q-1}{q^3(q-1)}\ge0
\end{align*}
where the last inequality follows from the fact that $x^3-2x^2+x-1$ is strictly increasing in $[q_f,2]$ and $(q_f(1))^3-2(q_f(1))^2+q_f(1)-1=0$.

%

(2)  For the case $(\textnormal{c}+\textnormal{d})|_1=0$ and $(\textnormal{c}+\textnormal{d}) \ge 0450^\infty$, we split the proof into two cases.

Case 1. $(\textnormal{c}+\textnormal{d})|_2=04$ and $\textnormal{c}+\textnormal{d}\ge0450^\infty$. Then
\begin{align*}
f_{\textnormal{c},\textnormal{d}}(q)&=((3+1)(c_i+d_i))_{q}-(3^{\infty})_{q}>(40450^{\infty})_q-(3^\infty)_q\\
&=\frac{4}{q}+\frac{4}{q^3}+\frac{5}{q^4}-\frac{3}{q-1}=\frac{q^4-4q^3+4q^2+q-5}{q^4(q-1)}.
\end{align*}
Now we need to verify  the numerator is positive for $q\in [q_f(3), 4]$.
Let $g(x)=x^4-4x^3+4x^2+x-5$. We have
\begin{align*}
 g'(x)=4x^3-12x^2+8x+1\;\textrm{and} \; g''(x)=4(3x^2-6x+2).
\end{align*}
Note that $q_f(3)$ is the largest real root of $x^3-3x^2+x-2=0$
($q_f(3)\approx 2.893$). Thus one can check that $g''(q_f(3)), g'(q_f(3)), g(q_f(3))$ and $g(4)$ are all positive. For instance, we have
$$
g(x)=x^4-4x^3+4x^2+x-5=(x-1)(x^3-3x^2+x-2)+4x-7.
$$
and so
$$
g(q_f(3))=4q_f(3)-7>0.
$$
Thus
we have $g(x)>0$ for $x\in [q_f(3), 4]$ by Proposition \ref{simlemma}.

%
%


Case 2.  If $(\textnormal{c}+\textnormal{d})|_1=0$ and $\textnormal{c}+\textnormal{d}\ge050^\infty$, then
\begin{align*}
f_{\textnormal{c},\textnormal{d}}(q)&=((3+1)(c_i+d_i))_{q}-(3^{\infty})_{q}>(4050^{\infty})_q-(3^\infty)_q\\
&=\frac{4}{q}+\frac{5}{q^3}-\frac{3}{q-1}=\frac{q^3-4q^2+5q-5}{q^3(q-1)}.
\end{align*}
We need to verify the numerator is positive for $q\in [q_f(3), 4]$.
Let $g(x)=x^3-4x^2+5x-5$. Note that $g'''(x)>0$ for $x\in [q_f(3), 4]$. And
\begin{align*}
g(q_f(3))&=(q_f(3))^3-4(q_f(3))^2+5q_f(3)-5\\
&=-(q_f(3))^2+4q_f(3)-3>0.
\end{align*}
Also it is easy to check that $g''(q_f(3)), g'(q_f(3))$ and $g(4)$ are positive.
  Thus  $g(x)>0$ for $x\in [q_f(3), 4]$ by Proposition \ref{simlemma}.


Now we discuss the case $k=1$ and $\textnormal{c}+\textnormal{d} \ge 120^\infty$.   Since $q\geq q_f(3)>2$
\begin{align*}
\min _{\alpha \ge 120^\infty }(\alpha )_q=\min \{(120^\infty )_q, (20^\infty )_q\}=(120^\infty )_q.
\end{align*}
Thus we have
\begin{align*}
f_{\textnormal{c},\textnormal{d}}(q)&=((3+1)(c_i+d_i))_{q}-(3^{\infty})_{q}
>(4120^{\infty})_q-(3^\infty)_q\\
&=\frac{4}{q}+\frac{1}{q^2}+\frac{2}{q^3}-\frac{3}{q-1}=\frac{q^3-3q^2+q-2}{q^3(q-1)}\ge0.
\end{align*}
The last inequality follows from the fact that  $x^3-3x^2+x-2$ is strictly increasing in $[q_f(3),4]$ and $(q_f(3))^3-3(q_f(3))^2+q_f(3)-2=0$ by (\ref{qfodd}).

(3) $k\ge2$. When  $(\textnormal{c}+\textnormal{d})|_1=k-1$ and $(\textnormal{c}+\textnormal{d})\ge(k-1)(m+2)0^\infty$, we have
\begin{align*}
f_{\textnormal{c},\textnormal{d}}(q)&=((m+1)(c_i+d_i))_{q}-(m^{\infty})_{q}\\
&>((m+1)(k-1)(m+2)0^{\infty})_q-(m^\infty)_q\\
&=\frac{m+1}{q}+\frac{k-1}{q^2}+\frac{m+2}{q^3}-\frac{m}{q-1}\\
&=\frac{q^3-(k+3)q^2+(k+4)q-2k-3}{q^3(q-1)}.
\end{align*}
In order to verify that the numerator is positive in $[q_f(2k+1), 2k+2]$,
let
$$g(x)=x^3-(k+3)x^2+(k+4)x-2k-3.$$
Clearly $g'''(x)$ satisfies the condition (I) of Proposition \ref{simlemma} in $[q_f(2k+1), 2k+2]$. It is easy to see that $g'(q_f(2k+1))>0$,
$g''(q_f(2k+1))>0$ and $g(2k+2)>0$. Note that
\begin{align*}
x^3-(k+3)x^2+(k+4)x-2k-3=&(x^3-(k+2)x^2+x-k-1)\\
&+ (-x^2+(k+3)x-k-2).
\end{align*}
Recall that $q_f(2k+1)$ is the  root of $x^3-(k+2)x^2+x-k-1=0$. Since $k+1<q_f(2k+1)<k+2$ we have
\begin{align*}
  g(q_f(2k+1))&=-(q_f(2k+1))^2+(k+3)q_f(2k+1)-k-2\\
  &=-(q_f(2k+1)-1)(q_f(2k+1)-k-2)>0.
\end{align*}
Therefore, $g(x)>0$ for  $x\in [q_f(2k+1), 2k+2]$ by Proposition \ref{simlemma}.

Now we consider the case $\textnormal{c}+\textnormal{d} \ge k(k+1)0^\infty$. Note that for $q\in (q_f(2k+1), 2k+2]$
\begin{align*}
\min _{\alpha \ge k(k+1)0^\infty}(\alpha )_q=\min \{(k(k+1)0^\infty )_q, ((k+1)0^\infty )_q\}=(k(k+1)0^\infty )_q.
\end{align*}
Thus, when  $\textnormal{c}+\textnormal{d}\ge k(k+1)0^\infty$,  we have
\begin{align*}
f_{\textnormal{c},\textnormal{d}}(q)&=((m+1)(c_i+d_i))_{q}-(m^{\infty})_{q}\\
&>((m+1)k(k+1)0^{\infty})_q-(m^\infty)_q\\
&=\frac{m+1}{q}+\frac{k}{q^2}+\frac{k+1}{q^3}-\frac{m}{q-1}=\frac{q^3-(k+2)q^2+q-k-1}{q^3(q-1)}.
\end{align*}
Let $g(x)=x^3-(k+2)x^2+x-k-1$. Then $g'(x)>0$ for $x\in (q_f(2k+1), 2k+2]$. So we have  $g(x)>0$  for $x\in (q_f(2k+1), 2k+2]$ since $g(q_f(2k+1))=0$.

(4) Suppose that $k=0$ and $\textnormal{c}+\textnormal{d}<00120^\infty$. Then $\textnormal{c}+\textnormal{d}\le 00112^\infty$. Now take $q_1,q_2\in[q,2]$ with $q_2>q_1$. It is important to point out that both
$f_{\textnormal{c},\textnormal{d}}(q_2)$ and $f_{\textnormal{c},\textnormal{d}}(q_1)$ make sense because  $\textnormal{c}, \textnormal{d}\in A'_t$ for all $t\in [q, 2]$.
Then we have
\begin{align*}
&f_{\textnormal{c},\textnormal{d}}(q_2)-f_{\textnormal{c},\textnormal{d}}(q_1)
= (2(c_i+d_i))_{q_2}-(1^{\infty})_{q_2}-[(2(c_i+d_i))_{q_1}-(1^{\infty})_{q_1}]\\
&=2\left (\frac{1}{q_2}-\frac{1}{q_1}\right )+\sum _{k=1}^\infty (c_k+d_k)\left (\frac{1}{q_2^{k+1}}-\frac{1}{q_1^{k+1}}\right )-(1^{\infty})_{q_2}+(1^{\infty})_{q_1}.
\end{align*}
Note that $\frac{1}{q_2^{k}}-\frac{1}{q_1^{k}}<0$ for all $k\geq 1$. Obviously for two sequences $(\alpha _i), (\beta _i)$ of nonnegative integers, if all $\alpha _i\leq
\beta _i$ then
$$
\sum _{k=1}^\infty \alpha _k\left (\frac{1}{q_2^{k}}-\frac{1}{q_1^{k}}\right )\geq \sum _{k=1}^\infty \beta _k\left (\frac{1}{q_2^{k}}-\frac{1}{q_1^{k}}\right ).
$$
Thus when $(\textnormal{c}+\textnormal{d})|_4= 0011$, we have
\begin{align*}
&2\left (\frac{1}{q_2}-\frac{1}{q_1}\right )+\sum _{k=1}^\infty (c_k+d_k)\left (\frac{1}{q_2^{k+1}}-\frac{1}{q_1^{k+1}}\right )\\
&\ge  (200112^\infty)_{q_2}-(200112^\infty)_{q_1}.
\end{align*}
Furthermore, when $(\textnormal{c}+\textnormal{d})|_4= 0002$ we claim
\begin{align*}
&2\left (\frac{1}{q_2}-\frac{1}{q_1}\right )+\sum _{k=1}^\infty (c_k+d_k)\left (\frac{1}{q_2^{k+1}}-\frac{1}{q_1^{k+1}}\right )\\
&\ge  (200112^\infty)_{q_2}-(200112^\infty)_{q_1}
\end{align*}
also holds. This is because we have
$$
\frac{1}{q_2^{5}}-\frac{1}{q_1^{5}}>\frac{1}{q_2^4}-\frac{1}{q_1^4}
$$
for $q_2>q_1$ with $q_1, q_2\in [q,2]\subseteq [q_f(1), 2]$. Therefore, we obtain that
$$
f_{\textnormal{c},\textnormal{d}}(q_2)-f_{\textnormal{c},\textnormal{d}}(q_1)
\ge  \left ((200112^\infty)_{q_2}-(1^{\infty})_{q_2}\right )-\left ((200112^\infty)_{q_1}-(1^{\infty})_{q_1}\right ).
$$
Now let us think about the function
\begin{align*}
g(x)&=(200112^\infty)_x-(1^{\infty})_x=\frac{2}{x}+\frac{1}{x^4}+\frac{1}{x^5}+\frac{2}{x^5(x-1)}-\frac{1}{x-1}\\
&=\frac{x^5-2x^4+x^2+1}{x^5(x-1)}.
\end{align*}
 We shall  prove that it is strictly increasing in $[q,2]$. Note that
 \begin{align*}
 g'(x)=\frac{-x^6+4x^5-2x^4-4x^3+3x^2-6x+5}{x^6(x-1)^2}:=\frac{h(x)}{x^6(x-1)^2}.
 \end{align*}
What  left is to verify that $h(x)>0$ for $x\in [q_f(1),2]$. We have
\begin{align*}
& h'(x)=2(-3x^5+10x^4-4x^3-6x^2+3x-3)\\
 & h''(x)=2(-15x^4+40x^3-12x^2-12x+3)\\
 & h'''(x)=24(-5x^3+10x^2-2x-1)\\
 & h^{(4)}(x)=24(-15x^2+20x-2).
\end{align*}
It is easy to know that $ h^{(4)}(x)<0$ for $x\in [q_f(1), 2]$, which implies that $h'''(x)$ is strictly decreasing in $[q_f(1), 2]$. Since
\begin{align*}
h'''(q_f(1))&=24(-5(q_f(1))^3+10(q_f(1))^2-2q_f(1)-1)\\
&=24(3q_f(1)-6)<0,
\end{align*}
we have $h'''(x)<0$ for $x\in [q_f(1),2]$. So $h''(x)$ is strictly decreasing in $[q_f(1), 2]$. Now $h''(2)=22$ and so $h''(x)>0$ for $x\in [q_f(1),2]$. Note that
\begin{align*}
h'(x)&=2(-3x^5+10x^4-4x^3-6x^2+3x-3)\\
&=2((x^3-2x^2+x-1)(-3x^2+4x+7)+x^2+4).
\end{align*}
Thus, $h'(q_f(1))=2(q_f(1))^2+8>0$.
Finally we have $h(2)>0$ and
\begin{align*}
h(x)&=-x^6+4x^5-2x^4-4x^3+3x^2-6x+5\\
&=(x^3-2x^2+x-1)(-x^3+2x^2+3x-1)-2x+4.
\end{align*}
Thus, $h(q_f(1))=-2q_f(1)+4>0$. Therefore, $h(x)>0$ for $x\in [q_f(1),2]$ by Proposition \ref{simlemma}.

%
%
%


(5) We first consider the case that $k=1$ and $\textnormal{c}+\textnormal{d} <0450^\infty$. We have $m=3$ and $\textnormal{c}+\textnormal{d}\in \Omega _6^{\mathbb N}$.
The condition $\textnormal{c}+\textnormal{d} <0450^\infty$ implies $\textnormal{c}+\textnormal{d}\le0446^\infty$.
Take $q_1,q_2\in[q,4]\subseteq [q_f(3), 4]$ with $q_2>q_1$. We have
\begin{align*}
&f_{\textnormal{c},\textnormal{d}}(q_2)-f_{\textnormal{c},\textnormal{d}}(q_1)= \left ((4(c_i+d_i))_{q_2}-(3^{\infty})_{q_2}\right )-\left ((4(c_i+d_i))_{q_1}-(3^{\infty})_{q_1}\right )\\
&=4\left (\frac{1}{q_2}-\frac{1}{q_1}\right )+\sum _{k=1}^\infty (c_k+d_k)\left (\frac{1}{q_2^{k+1}}-\frac{1}{q_1^{k+1}}\right )-(3^{\infty})_{q_2}+(3^{\infty})_{q_1}.
\end{align*}
When $(\textnormal{c}+\textnormal{d})|_3= 044$, using the same argument as that in (4) we have
\begin{align*}
&4\left (\frac{1}{q_2}-\frac{1}{q_1}\right )+\sum _{k=1}^\infty (c_k+d_k)\left (\frac{1}{q_2^{k+1}}-\frac{1}{q_1^{k+1}}\right )\\
&\ge  (40446^\infty)_{q_2}-(40446^\infty)_{q_1}.
\end{align*}
Furthermore, when $(\textnormal{c}+\textnormal{d})|_4= 036$ we claim
\begin{align*}
&4\left (\frac{1}{q_2}-\frac{1}{q_1}\right )+\sum _{k=1}^\infty (c_k+d_k)\left (\frac{1}{q_2^{k+1}}-\frac{1}{q_1^{k+1}}\right )\\
&\ge  (40446^\infty)_{q_2}-(40446^\infty)_{q_1}
\end{align*}
also holds. This is because we have
$$
2\left (\frac{1}{q_2^4}-\frac{1}{q_1^4}\right )>\frac{1}{q_2^3}-\frac{1}{q_1^3}
$$
for $q_2>q_1$ with $q_1, q_2\in [q,4]\subseteq [q_f(3), 4]$. We leave the verification for readers. Therefore, we obtain that
$$
f_{\textnormal{c},\textnormal{d}}(q_2)-f_{\textnormal{c},\textnormal{d}}(q_1)
\ge  \left ((40446^\infty)_{q_2}-(3^{\infty})_{q_2}\right )-\left ((40446^\infty)_{q_1}-(3^{\infty})_{q_1}\right ).
$$

Now let's take
$$
g(x)=(40446^\infty)_x-(3^{\infty})_x=\frac{x^4-4x^3+4x^2+2}{x^4(x-1)}
$$
 and show that $g(x)$ is strictly increasing in $[q,4]\subseteq [q_f(3), 4]$. We have
 $$
 g'(x)=\frac{-x^5+8x^4-16x^3+8x^2-10x+8}{x^5(x-1)^2}:=\frac{h(x)}{x^5(x-1)^2}.
 $$
 In order to show $h(x)>0$ for $x\in [q, 4]$, we take derivatives of $h(x)$
\begin{align*}
& h'(x)=-5x^4+32x^3-48x^2+16x-10\\
 & h''(x)=4(-5x^3+24x^2-24x+4)\\
 & h'''(x)=12(-5x^2+16x-8).
\end{align*}
Easy to check that $ h'''(x)<0$ for $x\in [q_f(3), 4]$. Recall that $q_f(3)$ satisfies
\begin{equation}\label{qfm}
(q_f(3))^3-3(q_f(3))^2+q_f(3)-2=0.
\end{equation}
Using (\ref{qfm}) one can show that  $h'(q_f(3))>0$.
In fact,
\begin{align*}
&-5x^4+32x^3-48x^2+16x-10\\
=&(x^3-3x^2+x-2)(-5x+17)+8x^2-11x+24.
\end{align*}

Thus by (\ref{qfm}) we have
\begin{align*}
h'(q_f(3))&=8(q_f(3))^2-11q_f(3)+24\\
&=8\left ((q_f(3))^2-3q_f(3)+1\right )+13q_f(3)+16\\
&=\frac{16}{q_f(3)}+13q_f(3)+16>0.
\end{align*}
Similarly one can get $h''(q_f(3))>0$ and  $h(q_f(3))>0$ by means of (\ref{qfm}).
Finally $h(4)=96>0$.
Hence $h(x)>0$ for $x\in [q_f(3), 4]\supseteq [q, 4]$ by Proposition \ref{simlemma}.


%
%

Now we turn to consider the case $k=1$,  $(\textnormal{c}+\textnormal{d})|_1=1$ and $\textnormal{c}+\textnormal{d} <120^\infty$. Note that $\textnormal{c}+\textnormal{d}\in \Omega ^{\mathbb N}_6$. Thus
 $\textnormal{c}+\textnormal{d}\le 116^\infty$. As before, for $q_1, q_2\in [q, 4]\subseteq [q_f(3),4]$ with $q_1<q_2$ we have
\begin{align*}
&f_{\textnormal{c},\textnormal{d}}(q_2)-f_{\textnormal{c},\textnormal{d}}(q_1)
= (4(c_i+d_i))_{q_2}-(3^{\infty})_{q_2}-\left ((4(c_i+d_i))_{q_1}-(3^{\infty})_{q_1}\right )\\
\ge&(4116^\infty)_{q_2}-(3^{\infty})_{q_2}-(4116^\infty)_{q_1}+(3^{\infty})_{q_1}.
\end{align*}
Again let
\begin{align*}
g(x)=(4116^\infty)_x-(3^{\infty})_x
=\frac{x^3-3x^2+5}{x^3(x-1)},
\end{align*}
and try to prove it is strictly increasing for $x\in [q,4]\subseteq [q_f(3),4]$. We have
\begin{align*}
 g'(x)=\frac{-x^4+6x^3-3x^2-20x+15}{x^4(x-1)}:=\frac{h(x)}{x^4(x-1)}.
\end{align*}
Using the same argument as above, one can verify that $h''(x)<0$ for $x\in [q_f(3),4]$ so that $h'(x)$ is strictly decreasing in $[q_f(3),4]$. But $h'(x)>0$ in $[q_f(3),\alpha )$ and
$h'(x)<0$ in $(\alpha , 4]$ for some $q_f(3)<\alpha <4$. Hence $h(x)>0$ in $[q_f(3),4]$ by the fact both $h(q_f(3))$ and $h(4)$ are positive.


(6) We first consider the case that $k\ge2$ and  $\textnormal{c}+\textnormal{d}<(k-1)(m+2)0^\infty$ where $m=2k+1$. Take $q_1,q_2\in[q,m+1]\subseteq [q_f(2k+1), 2k+2]$ with $q_2>q_1$. We split the proof into two cases.

Case 1.  $(\textnormal{c}+\textnormal{d})|_1<(k-1)$. Then $\textnormal{c}+\textnormal{d}\le(k-2)(2m)^\infty$. Hence
\begin{align*}
f_{\textnormal{c},\textnormal{d}}(q_2)-f_{\textnormal{c},\textnormal{d}}(q_1)&=((m+1)(c_i+d_i))_{q_2}-(m^{\infty})_{q_2}\\
&-((m+1)(c_i+d_i))_{q_1}+(m^{\infty})_{q_1}\\
&\ge ((m+1)(k-2)(2m)^\infty)_{q_2}-(m^{\infty})_{q_2}\\
\quad&-\left (((m+1)(k-2)(2m)^\infty)_{q_1}-(m^{\infty})_{q_1}\right ).
\end{align*}
As before we take
\begin{align*}
g(x)&=((m+1)(k-2)(2m)^\infty)_x-(m^{\infty})_x\\
&=\frac{x^2-(k+4)x+3k+4}{x^2(x-1)},
\end{align*}
and prove $g(x)$ is strictly increasing in $[q_f(2k+1), m+1]\supseteq [q,m+1]$. We have
\begin{align*}
 g'(x)=\frac{-x^3+(2k+8)x^2-(10k+16)x+6k+8}{x^3(x-1)^2}.
\end{align*}
 Let $h(x)$ be the numerator of $g'(x)$. We shall prove $h(x)>0$ in $[q_f(2k+1), m+1]=[q_f(2k+1), 2k+2]$. Note that
%
\begin{align*}
 h'(x)=-3x^2+4(k+4)x-10k-16.
\end{align*}
The roots $x_1<x_2$ of $h'(x)=0$ satisfy for $k\ge3$
\begin{align*}
&x_1=\frac{2k+8-\sqrt{4k^2+2k+16}}{3}<k+1<q_f(2k+1)\\
&k+1<x_2=\frac{2k+8+\sqrt{4k^2+2k+16}}{3}< m+1=2k+2
\end{align*}
and for $k=2$
\begin{align*}
x_1=2<k+1<q_f(5)<k+2<x_2=6=2k+2.
\end{align*}
Hence $h(x)$ is strictly increasing in $[k+1,x_2]$ and strictly decreasing in $[x_2,m+1]$ for $k\ge3$, and $h(x)$ is strictly increasing in $[q_f(2k+1), m+1]=[q_f(5), 6]$ for $k=2$. However, for $k\geq 3$
$$
h(k+1)=k^3-k^2-5k-1>0\;\textrm{and}\; h(m+1)=4k^2+2k>0
$$
and $h(q_f(5))>0$ (case $k=2$). Thus,  $h(x)>0$ in $[q_f(2k+1), 2k+2]$.


Case 2. $(\textnormal{c}+\textnormal{d})|_1=k-1$ and $\textnormal{c}+\textnormal{d}<(k-1)(m+2)0^\infty$. In this case we have
$\textnormal{c}+\textnormal{d}\le(k-1)(m+1)(2m)^\infty$. And so
\begin{align*}
f_{\textnormal{c},\textnormal{d}}(q_2)-f_{\textnormal{c},\textnormal{d}}(q_1)&=((m+1)(c_i+d_i))_{q_2}-(m^{\infty})_{q_2}\\
&-((m+1)(c_i+d_i))_{q_1}+(m^{\infty})_{q_1}\\
&\ge ((m+1)(k-1)(m+1)(2m)^\infty)_{q_2}-(m^{\infty})_{q_2}\\
\quad&-\left (((m+1)(k-1)(m+1)(2m)^\infty)_{q_1}-(m^{\infty})_{q_1}\right ).
\end{align*}
It suffices to prove $g(x)$ is strictly increasing in $[q_f(2k+1), 2k+2]$ where
\begin{align*}
g(x)&=((m+1)(k-1)(m+1)(2m)^\infty)_x-(m^{\infty})_x\\
&=\frac{x^3-(k+3)x^2+(k+3)x+2k}{x^3(x-1)}.
\end{align*}
We have
\begin{align*}
g'(x)&=\frac{-x^4+2(k+3)x^3-(4k+12)x^2-(6k-6)x+6k}{x^4(x-1)^2}\\
&:=\frac{h(x)}{x^4(x-1)^2}
\end{align*}

and
\begin{align*}
  &h'(x)=2(-2x^3+3(k+3)x^2-(4k+12)x-3k+3)\\
&h''(x)=-4(3x^2-3(k+3)x+2k+6)\\
&h'''(x)=-4(6x-3(k+3)).
\end{align*}
So $h'(k+1)=2(k^3+5k^2-4k-2)>0$, $h''(k+1)=16k>0$ and $h'''(x)<0$ in $[k+1,m+1]$. In addition,
\begin{align*}
 &h(k+1)=k^4+4k^3-8k^2-6k-1>0\\
 &h(m+1)=h(2k+2)=(8k^2-6k-2)(2k+2)+6k>0.
 \end{align*}
Hence $h(x)>0$ for $x\in [k+1,m+1]\supseteq [q_f(2k+1), 2k+2]$ by Proposition \ref{simlemma}.

%

Now we turn to consider the case that $k\geq 2$, $(\textnormal{c}+\textnormal{d})|_1=k$ and $\textnormal{c}+\textnormal{d}<k(k+1)0^\infty$.
Then $\textnormal{c}+\textnormal{d}\le kk(2m)^\infty$, and so
\begin{align*}
f_{\textnormal{c},\textnormal{d}}(q_2)-f_{\textnormal{c},\textnormal{d}}(q_1)&=((m+1)(c_i+d_i))_{q_2}-(m^{\infty})_{q_2}\\
&-((m+1)(c_i+d_i))_{q_1}+(m^{\infty})_{q_1}\\
&\ge ((m+1)kk(2m)^\infty)_{q_2}-(m^{\infty})_{q_2}\\
\quad&-\left (((m+1)kk(2m)^\infty)_{q_1}-(m^{\infty})_{q_1}\right ).
\end{align*}
Below we verify that  $g(x)$ is strictly increasing in $[q,m+1]\subseteq [q_f(2k+1), m+1]$ where
\begin{align*}
g(x)&=((m+1)kk(2m)^\infty)_x-(m^{\infty})_x
=\frac{x^3-(k+2)x^2+3k+2}{x^3(x-1)}.
\end{align*}
We have
\begin{align*}
g'(x)&=\frac{-x^4+(2k+4)x^3-(k+2)x^2-(12k+8)x+3(3k+2)}{x^4(x-1)^2}\\
&:=\frac{h(x)}{x^4(x-1)^2},
\end{align*}
and
\begin{align*}
 &h'(x)=2(-2x^3+3(k+2)x^2-(k+2)x-6k-4)\\
 &h''(x)=-2(6x^2-6(k+2)x+k+2)\\
 &h'''(x)=-12(2x-(k+2)).
\end{align*}
Then $h'''(x)<0$ in $[k+1,m+1]$, $h''(k+1)=-2(-5k-4)>0$ and $h'(k+1)=2(k^3+5k^2-2)>0$. Since
\begin{align*}
 &h(k+1)=(k^3+4k^2-8k-7)(k+1)+9k+6>0\\
 &h(m+1)=(6k^2-2k-4)(2k+2)+9k+6>0,
\end{align*}
 we have $h(x)>0$ for $x\in [k+1,m+1]\supseteq [q_f(2k+1), 2k+2]$ by Proposition \ref{simlemma}.
\end{proof}

We deduce from Lemma \ref{9} the following result.


 \begin{corollary}\label{c1}
Let $m=2k+1$, $q\in[q_f(m),m+1]$ and $\textnormal{c},\textnormal{d}\in A'_q$ w.r.t. $\Omega _m$. If $f_{\textnormal{c},\textnormal{d}}(q)\le0$, then  $f_{\textnormal{c},\textnormal{d}}(t)$ is strictly increasing in $[q,m+1]$.
\end{corollary}
Now we consider the case that $m$ is even.

\begin{lemma}\label{L8}
Let $m=2$, $q \in [q_f(2),3]$ and $\textnormal{c},\textnormal{d} \in A'_{q}$ w.r.t. $\Omega_{2}$.
\begin{enumerate}[\upshape (1)]
\item If $\textnormal{c}+\textnormal{d}\ge 0210^{\infty}$, then $f_{\textnormal{c},\textnormal{d}}(q)>0$.
\item 
    If $\textnormal{c}+\textnormal{d}< 0210^{\infty}$, then $f_{\textnormal{c},\textnormal{d}}(t)$ is strictly increasing for $t\in [q,3]$.
\end{enumerate}
\end{lemma}

\begin{proof}

(1)  We have $q_f(2)=1+\sqrt{2}$ by (\ref{qfeven}). For any $q\in [q_f(2), 3]$
\begin{align*}
\min _{\alpha \geq 0210^\infty }(\alpha )_q=\min \{(0210^\infty )_q, (10^\infty )_q, (030^\infty )_q\}=(0210^\infty )_q.
\end{align*}
Thus, when  $\textnormal{c}+\textnormal{d}\ge 0210^{\infty}$ we have
\begin{align*}
f_{\textnormal{c},\textnormal{d}}(q)&=(3(c_i+d_i))_{q}-(2^{\infty})_{q}>(30210^{\infty})_q-(2^\infty)_q\\
&=\frac{3}{q}+\frac{2}{q^3}+\frac{1}{q^4}-\frac{2}{q-1}=\frac{q^4-3q^3+2q^2-q-1}{q^4(q-1)}.
\end{align*}
We need to check  $g(x):=x^4-3x^3+2x^2-x-1\geq 0$ for $x\in [q_f(2), 3]$. Note that
\begin{align*}
g'(x)=4x^3-9x^2+4x-1 \;\textrm{and}\; g''(x)=12x^2-18x+4.
\end{align*}
 We have $g''(x)>0$ for $x\in [q_f(2), 3]$. Since $g'(q_f(2))=4+6\sqrt{2}$ so $g'(x)>0$ for $x\in [q_f(2), 3]$. Finally, it follows from $g(q_f(2))=0$ that $g(x)\ge 0$ for $x\in [q_f(2), 3]$.


(2) We now consider the case  $\textnormal{c}+\textnormal{d}< 0210^{\infty}$. Take $q_1,q_2\in[q,3]$ with $q_2>q_1$.
Suppose that $(\textnormal{c}+\textnormal{d})|_2\le01$. Then $(\textnormal{c}+\textnormal{d})<014^\infty$ and
\begin{align*}
f_{\textnormal{c},\textnormal{d}}(q_2)-f_{\textnormal{c},\textnormal{d}}(q_1)
=&(3(c_i+d_i))_{q_2}-(2^{\infty})_{q_2}-(3(c_i+d_i))_{q_1}+(2^{\infty})_{q_1}\\
\ge& (3014^\infty)_{q_2}-(2^{\infty})_{q_2}-\left (3014^\infty)_{q_1}-(2^{\infty})_{q_1}\right ).
\end{align*}
Let 
\begin{align*}
g(x)&=(3014^\infty)_x-(2^{\infty})_x
=\frac{x^3-3x^2+x+3}{x^3(x-1)}.
\end{align*}
In the same way as in Lemma \ref{9} one can show that $g(x)$ is strictly increasing in $[q_f(2), 3]$.

Furthermore, when $(\textnormal{c}+\textnormal{d})|_2=02$ and $\textnormal{c}+\textnormal{d}<0210^\infty$ we claim
\begin{align*}
f_{\textnormal{c},\textnormal{d}}(q_2)-f_{\textnormal{c},\textnormal{d}}(q_1)
=&(3(c_i+d_i))_{q_2}-(2^{\infty})_{q_2}-(3(c_i+d_i))_{q_1}+(2^{\infty})_{q_1}\\
\ge& (3014^\infty)_{q_2}-(2^{\infty})_{q_2}-\left (3014^\infty)_{q_1}-(2^{\infty})_{q_1}\right )
\end{align*}
also holds. This is because we have
$$
\frac{1}{q_2^3}-\frac{1}{q_1^3}>\frac{4}{q_2^4}-\frac{4}{q_1^4}
$$
for $q_2>q_1$ with $q_1, q_2\in [q,2]\subseteq [q_f(2), 3]$.

\end{proof}

\begin{lemma}\label{8}
Let $m=2k$, $q \in [q_f(m),m+1]$ and $\textnormal{c},\textnormal{d} \in A'_{q}$ over $\Omega_{m}$.
\begin{enumerate}[\upshape (1)]
\item Let $k\ge 1$. If  $\textnormal{c}+\textnormal{d}\ge k0^\infty$, then $f_{\textnormal{c},\textnormal{d}}(q)>0$.
\item Let $k\ge2$. We have $f_{\textnormal{c},\textnormal{d}}(q)> 0$ if $\textnormal{c},\textnormal{d}$ satisfy  one of the following conditions:

 (i) $(\textnormal{c}+\textnormal{d})|_1=k-1$ and $\textnormal{c}+\textnormal{d}\ge (k-1)(m-1)(k+1)0^\infty$;

 (ii) $(\textnormal{c}+\textnormal{d})|_1=k-2$ and $\textnormal{c}+\textnormal{d} \ge (k-2)(2m-1)(k+1)0^\infty$.

\item   Let $k\ge2$. $f_{\textnormal{c},\textnormal{d}}(t)$ is strictly increasing in $[q,m+1]$ if $\textnormal{c},\textnormal{d}$ satisfy  one of the following conditions:

 (i) $(\textnormal{c}+\textnormal{d})|_1=k-1$ and $\textnormal{c}+\textnormal{d}<(k-1)(m-1)(k+1)0^\infty$;

 (ii) $\textnormal{c}+\textnormal{d}<(k-2)(2m-1)(k+1)0^\infty$.
\end{enumerate}
\end{lemma}

\begin{proof}

(1)  When $(\textnormal{c}+\textnormal{d})\ge k0^{\infty}$ we have
\begin{align*}
f_{\textnormal{c},\textnormal{d}}(q)&=((m+1)(c_i+d_i))_{q}-(m^{\infty})_{q}\\
&>((m+1)k0^{\infty})_q-(m^\infty)_q\\
&=\frac{m+1}{q}+\frac{k}{q^2}-\frac{m}{q-1}=\frac{q^2-(k+1)q-k}{q^2(q-1 )}\ge 0.
\end{align*}
The last equality follows from the fact that $x^2-(k+1)x-k$ is strictly increasing in $[q_f(m),m+1]$ and $(q_{f}(2k))^2-(k+1)q_{f}(2k)-k=0$.

(2) We first consider the case (i):  $(\textnormal{c}+\textnormal{d})|_1=k-1$ and $\textnormal{c}+\textnormal{d}\ge(k-1)(m-1)(k+1)0^{\infty}$. Note that
for any $q\in [q_f(2k), 2k+1]$
\begin{align*}
\min _{\alpha \geq (m-1)(k+1)0^{\infty} }(\alpha )_q&=\min \{((m-1)(k+1)0^{\infty})_q, (m0^{\infty})_q\}\\
&=((m-1)(k+1)0^{\infty})_q.
\end{align*}
 Thus we have
\begin{align*}
f_{\textnormal{c},\textnormal{d}}(q)&=((m+1)(c_i+d_i))_{q}-(m^\infty)_q\\
&\ge((m+1)(k-1)(m-1)(k+1)0^{\infty})_q-(m^\infty)_q\\
&=\frac{q^4-(k+2)q^3+kq^2-(k-2)q-k-1}{q^4(q-1 )}.
\end{align*}
Let  $g(x)=x^4-(k+2)x^3+kx^2-(k-2)x-k-1$ and we shall show it is positive for $x \in [q_f(2k), 2k+1]$. We have
\begin{align*}
&g'(x)=4x^3-3(k+2)x^2+2kx-k+2 \\
&g''(x)=2(6x^2-3(k+2)x+k).
\end{align*}
The roots of $g''(x)$ satisfy
\begin{align*}
x_1<x_2=\frac{3k+6+\sqrt{9k^2+12k+36}}{12}<k+1<q_f(2k)
\end{align*}
and $g'(k+1)=k^3+2k^2-2k>0$. Hence $g(x)$ is strictly increasing in $[q_f(2k), 2k+1]$. Next we shall show that  $g(q_f(2k))>0$ for $k\ge2$.
 Recall $q_f(2k)=(k+1+\sqrt{k^2+6k+1})/2$ is the largest real root of $x^2-(k+1)x-k=0$ and
 \begin{align*}
g(x)&=x^4-(k+2)x^3+kx^2-(k-2)x-k-1\\
&=(x^2-x+k-1)(x^2-(k+1)x-k)+(k-1)^2(x+1)-2.
\end{align*}
Thus $g(q_f(2k))>0$ for $k\ge2$.



Now we consider the  case (ii):  $(\textnormal{c}+\textnormal{d})|_1=k-2$ and $\textnormal{c}+\textnormal{d} \ge (k-2)(2m-1)(k+1)0^\infty$. Note that
for any $q\in [q_f(2k), 2k+1]$
\begin{align*}
\min _{\alpha \geq (2m-1)(k+1)0^{\infty} }(\alpha )_q&=\min \{((2m-1)(k+1)0^{\infty})_q, (2m0^{\infty})_q\}\\
&=((2m-1)(k+1)0^{\infty})_q.
\end{align*}
Thus we have
\begin{align*}
f_{\textnormal{c},\textnormal{d}}(q)&=((m+1)(c_i+d_i))_q-(m^\infty)_q\\
&\ge((m+1)(k-2)(2m-1)(k+1)0^{\infty})_q-(m^\infty)_q\\
&=\frac{q^4-(k+3)q^3+(3k+1)q^2-(3k-2)q-k-1}{q^4(q-1 )}.
\end{align*}
Let $g(x)=x^4-(k+3)x^3+(3k+1)x^2-(3k-2)x-k-1$. We have
\begin{align*}
&g'(x)=4x^3-3(k+3)x^2+2(3k+1)x-3k+2 \\
&g''(x)=2(6x^2-3(k+3)x+3k+1).
\end{align*}
The roots of $g''(x)$ satisfy for $k\ge2$
\begin{align*}
x_1<x_2=\frac{3k+9+\sqrt{9k^2-18k+57}}{12}<k+1<q_f(2k).
\end{align*}
Thus $g'(x)>0$ in $[k+1, 2k+1]$ by $g'(k+1)>0$. So $g(x)$ is strictly increasing in $[q_f(2k), 2k+1]$. Finally
\begin{align*}
g(x)&=x^4-(k+3)x^3+(3k+1)x^2-(3k-2)x-k-1\\
&=(x^2-2x+2k-1)(x^2-(k+1)x-k)\\
&\;\;\;+(2k^2-4k+1)x+(2k^2-2k-1).
\end{align*}
So $g(q_f(2k))>0$ for $k\ge2$.


(3) Let $q_1,q_2\in[q,m+1]=[q, 2k+1]$ with $q_2>q_1$.

(i) Suppose that $(\textnormal{c}+\textnormal{d})|_1=k-1$ and $\textnormal{c}+\textnormal{d}<(k-1)(m-1)(k+1)0^\infty$.
We split the proof into two cases.

Case 1.  $(\textnormal{c}+\textnormal{d})|_1=k-1$ and $(\textnormal{c}+\textnormal{d})<(k-1)(m-1)0^\infty$. Then $(\textnormal{c}+\textnormal{d})\le(k-1)(m-2)(2m)^\infty$. Thus
\begin{align*}
f_{\textnormal{c},\textnormal{d}}(q_2)-f_{\textnormal{c},\textnormal{d}}(q_1)&=((m+1)(c_i+d_i))_{q_2}-(m^{\infty})_{q_2}\\
&-((m+1)(c_i+d_i))_{q_1}+(m^{\infty})_{q_1}\\
&\ge ((m+1)(k-1)(m-2)(2m)^\infty)_{q_2}-(m^{\infty})_{q_2}\\
\quad&-\left (((m+1)(k-1)(m-2)(2m)^\infty)_{q_1}-(m^{\infty})_{q_1}\right ).
\end{align*}
We shall prove $g(x)$  is strictly increasing in $[q,m+1]$ where
\begin{align*}
g(x)&=((m+1)(k-1)(m-2)(2m)^\infty)_x-(m^{\infty})_x\\
&=\frac{x^3-(k+2)x^2+(k-1)x+2k+2}{x^3(x-1)}.
\end{align*}
Note that 
\begin{align*}
g'(x)&=\frac{-x^4+(2k+4)x^3-(4k-1)x^2-(6k+10)x+6(k+1)}{x^4(x-1)^2}\\
&:=\frac{h(x)}{x^4(x-1)^2},
\end{align*}
and
\begin{align*}
 &h'(x)=2(-2x^3+3(k+2)x^2-(4k-1)x-3k-5)\\
 &h''(x)=2(-6x^2+6(k+2)x-4k+1).
\end{align*}
The roots $x_1,x_2$ of $h''(x)$  satisfy for $k\ge2$
\begin{align*}
x_1<k+1<x_2=\frac{3k+6+\sqrt{9k^2+12k+42}}{6}<k+2.
\end{align*}
Hence $h'(x)$ is strictly increasing in $[k+1,x_2]$ and strictly decreasing in $[x_2,m+1]$ for $k\ge2$. Note that
$$
h'(k+1)=2(k^3+2k^2+3k)>0.
$$
In addition, we have $h'(m+1)=h'(2k+1)$ is positive when $k=2$, but negative for $k\geq 3$. This means that in $[k+1, m+1]$ the function $h(x)$ is strictly increasing when $k=2$, and $h(x)$ is strictly increasing firstly and then strictly decreasing when $k\geq 3$. However, we have
\begin{align*}
 &h(k+1)=(k^3+k^2-2k)(k+1)>0\\
 &h(m+1)=(4k^2+4k-6)(2k+1)+6k+6>0.
\end{align*}
Therefore,  $h(x)>0$ for $x\in [q_f(2k),m+1]$.

Case 2.  $(\textnormal{c}+\textnormal{d})|_2=(k-1)(m-1)$ and $(\textnormal{c}+\textnormal{d})< (k-1)(m-1)(k+1)0^\infty$. Then $(\textnormal{c}+\textnormal{d})\le (k-1)(m-1)k(2m)^\infty$. Thus
\begin{align*}
f_{\textnormal{c},\textnormal{d}}(q_2)-f_{\textnormal{c},\textnormal{d}}(q_1)&=((m+1)(c_i+d_i))_{q_2}
-(m^{\infty})_{q_2}\\
&\;\;\;-((m+1)(c_i+d_i))_{q_1}+(m^{\infty})_{q_1}\\
&\ge ((m+1)(k-1)(m-1)k(2m)^\infty)_{q_2}-(m^{\infty})_{q_2}\\
\quad& \;\;\;-\left (((m+1)(k-1)(m-1)k(2m)^\infty)_{q_1}-(m^{\infty})_{q_1}\right ).
\end{align*}
We shall show that  $g(x)$ is strictly increasing in $[q,m+1]$, where
\begin{align*}
g(x)&=((m+1)(k-1)(m-1)k(2m)^\infty)_x-(m^{\infty})_x\\
&=\frac{x^4-(k+2)x^3+kx^2-(k-1)x+3k}{x^4(x-1)}.
\end{align*}
Note that $g'(x)=\frac{h(x)}{x^5(x-1)^2}$ where
\begin{align*}
h(x)
&=-x^5+(2k+4)x^4-(4k+2)x^3\\
&+(6k-4)x^2-(18k-3)x+12k.
\end{align*}
Easy to check that $h'''(x)=12(-5x^2+4(k+2)x-2k-1)<0$ for $x\in [k+1, m+1]$. One can check that
$$
h''(k+1)>0, h'(k+1)>0, h(k+1)>0\;\textrm{ and}\;  h(m+1)>0.
$$
Therefore, $h(x)>0$ for $x\in [k+1, m+1]\supseteq [q_f(2k),m+1]$ by Proposition \ref{simlemma}.

(ii) Suppose that $(\textnormal{c}+\textnormal{d})<(k-2)(2m-1)(k+1)0^\infty$. We split the proof into three cases.

Case 1.  $(\textnormal{c}+\textnormal{d})|_2=(k-2)(2m-1)$ and  $(\textnormal{c}+\textnormal{d})<(k-2)(2m-1)(k+1)0^\infty$. Then $(\textnormal{c}+\textnormal{d})\le(k-2)(2m-1)k(2m)^\infty$. Thus
\begin{align*}
f_{\textnormal{c},\textnormal{d}}(q_2)-f_{\textnormal{c},\textnormal{d}}(q_1)
&=((m+1)(c_i+d_i))_{q_2}-(m^{\infty})_{q_2}\\
&\;\;\; -((m+1)(c_i+d_i))_{q_1}+(m^{\infty})_{q_1}\\
&\ge  ((m+1)(k-2)(2m-1)k(2m)^\infty)_{q_2}-(m^{\infty})_{q_2}\\
 & \;\;\; -\left (((m+1)(k-2)(2m-1)k(2m)^\infty)_{q_1}-(m^{\infty})_{q_1}\right ).
\end{align*}
As before, let $g(x)=((m+1)(k-2)(2m-1)k(2m)^\infty)_x-(m^{\infty})_x$ and try  to prove it is strictly increasing in $[q,m+1]$. Note that
\begin{align*}
g(x)&=((m+1)(k-2)(2m-1)k(2m)^\infty)_x-(m^{\infty})_x\\
&=\frac{x^4-(k+3)x^3+(3k+1)x^2-(3k-1)x+3k}{x^4(x-1)},
\end{align*}
and
\begin{align*}
g'(x)
&=\frac{-x^5+(2k+6)x^4-(10k+6)x^3
+(18k-2)x^2-(24k-3)x+12k}{x^5(x-1)^2}\\
&:=\frac{h(x)}{x^5(x-1)^2}.
\end{align*}
It is easy to see that $h'''(x)=12(-5x^2+4(k+3)x-5k-3)<0$ for $x\in  [q_f(2k), 2k+1]$. Moreover, one can check that
$$
h''(q_f(2k))>0, h'(q_f(2k))>0, h(q_f(2k))>0\; \textrm{ and}\; h(m+1)>0.
$$
 From Proposition \ref{simlemma}  it follows that $h(x)>0$ for
$x\in  [q_f(2k), 2k+1]$.

Case 2. $(\textnormal{c}+\textnormal{d})|_1=k-2$ and $(\textnormal{c}+\textnormal{d})<(k-2)(2m-1)0^\infty$. Then $(c_i+d_i)\le(k-2)(2m-2)(2m)^\infty$, and
\begin{align*}
f_{\textnormal{c},\textnormal{d}}(q_2)-f_{\textnormal{c},\textnormal{d}}(q_1)&=((m+1)(c_i+d_i))_{q_2}-(m^{\infty})_{q_2}\\
&\;\;\;-((m+1)(c_i+d_i))_{q_1}+(m^{\infty})_{q_1}\\
&\ge ((m+1)(k-2)(2m-2)(2m)^\infty)_{q_2}-(m^{\infty})_{q_2}\\
\quad&\;\;\;-\left (((m+1)(k-2)(2m-2)(2m)^\infty)_{q_1}-(m^{\infty})_{q_1}\right ).
\end{align*}
In the following we show that $g(x)$ is strictly increasing in $[q_f(2k), m+1]$ where
\begin{align*}
g(x)&=((m+1)(k-2)(2m-2)(2m)^\infty)_x-(m^{\infty})_x\\
&=\frac{x^3-(k+3)x^2+3kx+2}{x^3(x-1)}.
\end{align*}
Note that
\begin{align*}
g'(x)&=\frac{-x^4+(2k+6)x^3-(10k+3)x^2+(6k-8)x+6}{x^4(x-1)^2}\\
&:=\frac{h(x)}{x^4(x-1)^2},
\end{align*}
and
\begin{align*}
 &h'(x)=2(-2x^3+3(k+3)x^2-(10k+3)x+3k-4)\\
 &h''(x)=2(-6x^2+6(k+3)x-10k-3).
\end{align*}
We have $h'''(x)=-24x+12(k+3)<0$ in $[k+1, m+1]$. In addition, one can check
$$
h''(k+1)>0, h'(k+1)>0, h(k+1)\geq 0\;\textrm{and}\; h(m+1)>0.
$$
 From Proposition \ref{simlemma}  it follows that $h(x)>0$ for
$x\in  [q_f(2k), 2k+1]$. Here we like to remark that $ h(k+1)=0$ for $k=2$ however $ h(k+1)>0$ for $k\geq 3$. This does not affect us to
get the result in  $[q_f(2k), 2k+1]\subsetneqq [k+1, 2k+1]$.

Case 3.  $k\ge3$ and $(\textnormal{c}+\textnormal{d})<(k-2)0^\infty$. Then $(\textnormal{c}+\textnormal{d})\le(k-3)(2m)^\infty$. So
\begin{align*}
f_{\textnormal{c},\textnormal{d}}(q_2)-f_{\textnormal{c},\textnormal{d}}(q_1)
&=((m+1)(c_i+d_i))_{q_2}-(m^{\infty})_{q_2}\\
&\;\;\;-((m+1)(c_i+d_i))_{q_1}+(m^{\infty})_{q_1}\\
&\ge  ((m+1)(k-3)(2m)^\infty)_{q_2}-(m^{\infty})_{q_2}\\
\quad &\;\;\;-\left ((m+1)(k-3)(2m)^\infty)_{q_1}-(m^{\infty})_{q_1}\right ).
\end{align*}
Again one needs to prove $g(x)$ is strictly increasing in $[q_f(2k),m+1]$ where
\begin{align*}
g(x)&=((m+1)(k-3)(2m)^\infty)_x-(m^{\infty})_x\\
&=\frac{x^2-(k+4)x+3k+3}{x^2(x-1)}.
\end{align*}
Note that 
\begin{align*}
g'(x)&=\frac{-x^3+(2k+8)x^2-(10k+13)x+6(k+1)}{x^3(x-1)^2}\\
&:=\frac{h(x)}{x^3(x-1)^2}.
\end{align*}
The roots  $x_1,x_2$ of $h'(x)=-3x^2+4(k+4)x-10k-13=0$ satisfy
\begin{align*}
&x_1<q_f(6)<7<x_2\quad\text{for $k=3$}\\
&x_1<q_f(2k)<x_2<2k+1\quad\text{for $k\ge4$}.
\end{align*}
Moreover one can check that $h(q_f(2k))>0$ and $h(2k+1)>0$. Thus
 $h(x)>0$ in $[q_f(2k), 2k+1]$.
\end{proof}

Combing Lemma \ref{L8} with \ref{8} we have

\begin{corollary}\label{c2}
Let $m=2k$, $q\in[q_f(m),m+1]$ and $\textnormal{c}, \textnormal{d}\in A'_q$. If $f_{\textnormal{c},\textnormal{d}}(q)\le0$, then  $f_{\textnormal{c},\textnormal{d}}(t)$ is strictly increasing in $[q,m+1]$.
\end{corollary}


For $q \in [q_f(m),m+1]$ let
\begin{align*}
B'_q=\set{(\textnormal{c},\textnormal{d}): \textnormal{c},\textnormal{d}\in A'_q, f_{\textnormal{c},\textnormal{d}}(q)\le 0}.
\end{align*}
Then $B'_q\neq\emptyset $ follows from the following calculation
\begin{align*}
f_{0^\infty , 0^\infty }(q)&=((m+1)0^\infty)_q-(m^\infty)_q\\
&=\frac{1}{q}-\frac{m}{q(q-1)}\le0.
\end{align*}
On the other hand, $f_{\textnormal{c},\textnormal{d}}(m+1)\ge0$ always holds for any $(\textnormal{c},\textnormal{d})\in B'_q$.
From
Corollaries \ref{c1} and \ref{c2} it follows that for any $(\textnormal{c},\textnormal{d})\in B'_q$
there exists a unique $q_{\textnormal{c},\textnormal{d}}\in [q, m+1]$ such that
\begin{equation*}
f_{\textnormal{c},\textnormal{d}}(q_{\textnormal{c},\textnormal{d}})=(1\textnormal{c})_{q_{\textnormal{c},\textnormal{d}}}+(m\textnormal{d})_{q_{\textnormal{c},\textnormal{d}}}-
(m^{\infty})_{q_{\textnormal{c},\textnormal{d}}}=0,
\end{equation*}
which means  $q_{\textnormal{c},\textnormal{d}}\in\bb_2(m)$.

\begin{lemma}\label{10}
Let $q\in[q_f(m),m+1]$.
\begin{enumerate}[\upshape (1)]
\item\label{101}  If $(\textnormal{c},\textnormal{d}), (\textnormal{e},\textnormal{d})\in B'_q$ with $\textnormal{e}>\textnormal{c}$, then $q_{\textnormal{e},\textnormal{d}}<q_{\textnormal{c},\textnormal{d}}$.
\item If $(\textnormal{c},\textnormal{d}), (\textnormal{c},\textnormal{e})\in B'_q$ with $\textnormal{e}>\textnormal{d}$, then $q_{\textnormal{c},\textnormal{e}}<q_{\textnormal{c},\textnormal{d}}$.

\end{enumerate}
\end{lemma}
\begin{proof}
By the symmetry  it suffices to prove \eqref{101}. By Lemma \ref{6}
\begin{align*}
f_{\textnormal{c},\textnormal{d}}(q_{\textnormal{e},\textnormal{d}})&=(1\textnormal{c})_{q_{\textnormal{e},\textnormal{d}}}+(m\textnormal{d})_{q_{\textnormal{e},\textnormal{d}}}-(m^{\infty})_{q_{\textnormal{e},\textnormal{d}}}\\
&<(1\textnormal{e})_{q_{\textnormal{e},\textnormal{d}}}+(m\textnormal{d})_{q_{\textnormal{e},\textnormal{d}}}-(m^{\infty})_{q_{\textnormal{e},\textnormal{d}}}=f_{\textnormal{e},\textnormal{d}}(q_{\textnormal{e},\textnormal{d}})=0.
\end{align*}
By Corollaries \ref{c1} and \ref{c2} we have $q_{\textnormal{c},\textnormal{d}}>q_{\textnormal{e},\textnormal{d}}$.
\end{proof}

\section{Proof of Theorems \ref{T3} and \ref{T1}}
We firstly give a topological description of $\bb_2(m)$.
\begin{lemma}\label{21}
$\bb_2(m)$ is compact.
\end{lemma}
\begin{proof}
 \cite{KLZ} shows that $\bb_{2}(m)\cap(1,q_{f}(m))$ contains only finitely many points. We only need to consider $\bb_{2}(m)\cap[q_{f}(m),m+1]$ and to prove its complement is open.

We take $q\in[q_{f}(m),m+1]\setminus\bb_{2}(m)$ arbitrarily. So $1\notin\uu_q-\uu_q$ by Theorem 3.3.  The following lemma shows $q\notin\overline{\uu}$. Thus  $\uu_q-\uu_q$ is compact by   Lemma \ref{19} (i).
Then $d_H(1,\uu_q-\uu_q)>0$, where $d_H$ denotes the Hausdorff metric. Since $\uu_p$  continuously depends on $p\notin\overline{\uu}$ (see \cite{CK}), we take $0<\delta<d_H(1,\uu_q-\uu_q)$ and a small open set $O(\delta)$ which contains $q$  such that $d_H(\uu_{p_0}-\uu_{p_0},\uu_q-\uu_q)<\delta$ for all $p_0\in O(\delta )$. Then $d_H(1,\uu_{p_0}-\uu_{p_0})>0$, i.e. $p_0\notin\bb_2(m)$.
\end{proof}

  The following results reveal that there is a close connection between $\overline{\uu},\vv$ and $\bb_2(m)$.

\begin{lemma}\label{12}
$\overline{\uu} \subset \bb_2(m)$. Furthermore, $\overline{\uu}\subset\bb_2^{(\infty)}(m)$.
\end{lemma}
\begin{proof}
Take a $q\in \mathcal U$ arbitrarily, $1$ has a unique $q$-expansion, say $(c_i)$. Then $1=(c_i)_q-(0^\infty )_q$ which implies
 $q\in \bb_2(m)$ by Theorem \ref{T4}.
Since $\overline{\uu}$ is a Cantor set (\cite[Theorem 1.2]{VKL}),  we have $\overline{\uu}\subset\bb_2^{(\infty)}(m)$
when $\overline{\uu} \subset \bb_2(m)$ . Next we show that $\overline{\uu}\setminus\uu\subset\bb_2(m)$.

Let $q\in\overline{\uu}\setminus\uu$. By Lemma \ref{18} (ii) there exists a word $a_1a_2 \cdots a_n$ such that $\alpha(q)=(a_1a_2\cdots a_n)^{\infty}$ where $n$ is the smallest period of $\alpha(q)$. If $n=1$, then $\alpha(q)=(\alpha_1(q))^{\infty}$, which implies that $q=\alpha_1(q)+1$. Otherwise $q$ is a noninteger. Therefore we distinguish two cases.

\emph{Case I}: $q$ is a noninteger. In this case $n\geq2$ and $\beta(q)=a_1\cdots a_n^+0^\infty$. From Lemmas \ref{161} and \ref{17} (ii) we know that
\begin{equation}\label{e1}
\overline{a_1\cdots a_{n-i}}\le a_{i+1} \cdots a_n<a_{i+1} \cdots a_n^+\le a_1\cdots a_{n-i}
\end{equation}
for all $0<i<n$. Since $q\in\overline{\uu}\setminus\uu$, Lemma \ref{18} (i) tells us that  there exists a $p\in\uu\cap(1,q)$ such that
\begin{equation*}\label{e2}
\alpha_1(p)\cdots\alpha_n(p)=a_1a_2\cdots a_n.
\end{equation*}
Let
\begin{equation*}
{\textnormal{c}}=\overline{a_1\cdots a_{n}^+}\alpha(p), \qtq{and}{\textnormal{d}}=0^n\overline{\alpha(p)}.
\end{equation*}
It remains to prove that ${\textnormal{c}}, {\textnormal{d}} \in A'_q$ and $q=q_{{\textnormal{c}},{\textnormal{d}}}$.

First we prove that ${\textnormal{c}}, {\textnormal{d}}\in A'_q$. Since $p\in\uu\cap(1,q)$, by Lemmas \ref{16} and \ref{17} (i) we have
\begin{equation*}\label{20}
\overline{\alpha(q)}<\overline{\alpha(p)}<\sigma^i(\alpha(p))<\alpha(p)<\alpha(q)
\end{equation*}
for all $i>0$. Moreover, $a_1\ge k+1$ implies that $a_1>\overline{a_1}$. Then ${\textnormal{d}}\in A'_q$. On the other hand, for $0<i<n$ by \eqref{e1} we have $\overline{a_{i+1}\cdots a_n^+}\ge\overline{a_1\cdots a_{n-i}}$ and $a_1\cdots a_{i}\ge\overline{a_{n-i+1}\cdots a_n}$, which imply that
\begin{align*}
\overline{a_{i+1}\cdots a_n^+}\alpha(p)&=\overline{a_{i+1}\cdots a_n^+}a_{1}\cdots a_i \alpha_{i+1}(p)\cdots\\
&\ge\overline{a_{1}\cdots a_{n-i}}\overline{a_{n-i+1}\cdots a_n}\overline{\alpha(p)}\\
&>\overline{a_{1}\cdots a_n}\overline{\alpha(q)}=\overline{\alpha(q)}.
\end{align*}
Together with $\overline{a_{i+1}\cdots a_{n}^+}<a_1\cdots a_{n-i}$ we have $\textnormal{c}\in A'_q$. We conclude that $q=q_{{\textnormal{c}},{\textnormal{d}}}\in\bb_2(m)$ by the following calculation
\begin{align*}
f_{{\textnormal{c}}, {\textnormal{d}}}(q)&=(1{\textnormal{c}})_q+(m{\textnormal{d}})_q-(m^\infty)_q\\
&=(1\overline{a_1\cdots a_{n}^+}\alpha(p))_q+(m0^n\overline{\alpha(p)})_q-(m^\infty)_q\\
&=((m+1)\overline{a_1\cdots a_{n}^+}0^\infty)_q-(m^{n+1}0^\infty)_q\\
&=(10^\infty)_q-(0a_1\cdots a_{n}^+0^\infty)_q=0.
\end{align*}
\emph{Cases II}: $q$ is an integer. Let
\begin{equation*}
{\textnormal{c}}=\overline{a_1^+}\alpha(p)\qtq{and}{\textnormal{d}}=0\overline{\alpha(p)}.
\end{equation*}
As was the case in previous analysis,  it can be proved similarly.
\end{proof}

\begin{lemma}\label{13}
$\mathcal V\setminus\set{\mathcal G(m)}\subset\mathcal B_2(m)$.
\end{lemma}
\begin{proof}
Thanks to Lemma \ref{12} it suffices to prove that $\mathcal V\setminus(\overline{\uu}\cup\set{\mathcal G(m)})\subset\mathcal B_2(m)$. Given arbitrarily
 $q\in \mathcal V \setminus(\overline{\uu} \cup\set{\mathcal G(m)})$, Lemma \ref{18} (iii) tells us that
 there exists a word $a_1\cdots a_n$ with $n\ge1$ such that
\begin{equation*}
\alpha(q)=(a_1\cdots a_n^+\overline{a_1\cdots a_n^+})^\infty
\end{equation*}
and for all $0<i<n$
\begin{equation*}
\overline{(a_1\cdots a_n)^\infty}\le \sigma^i((a_1\cdots a_n)^\infty)<(a_1\cdots a_n)^\infty.
\end{equation*}
Take
$
{\textnormal{c}}=(c_i)=\overline{a_1\cdots a_n^+}(a_1\cdots a_n)^\infty $ and ${\textnormal{d}}=(d_i)=0^{2n}(\overline{a_1\cdots a_n})^\infty.$
It remains to show ${\textnormal{c}}, {\textnormal{d}}\in A'_q$ and $q=q_{{\textnormal{c}},{\textnormal{d}}}\in\bb_2(m)$.

(i) Suppose that $n\ge2$. Then by the above inequality and the definition of $\vv$ we have for all $0<i<n$
\begin{align}
&\overline{a_{i+1}\cdots a_n^+}<\overline{a_{i+1}\cdots a_n}\le a_{1}\cdots a_{n-i}\label{e42}\\
&a_{i+1}\cdots a_n<a_{i+1}\cdots a_n^+\le a_{1}\cdots a_{n-i}\label{e43}.
\end{align}
By Lemma \ref{L22} it suffices to prove for all $0<i<2n$
\begin{align*}
&c_{i+1} c_{i+2}\cdots<\alpha(q)\quad\text{when $c_i<m$}\\
&\overline{c_{i+1} c_{i+2}\cdots}<\alpha(q)\quad\text{when $c_i>0$}.
\end{align*}
If $0<i<n$, we use \eqref{e42} and \eqref{e43}. If $n<i<2n$, we use \eqref{e42} and \eqref{e43} once again.  For the case $i=n$, we have $a_1\ge k+1$ ($m=2k$ or $m=2k+1$) it follows from $\overline{a_1}<a_1$. Hence, ${\textnormal{c}}\in A^{\prime}_q$. Similarly one can show  ${\textnormal{d}} \in A^{\prime}_q$.

Next we show $q=q_{{\textnormal{c}},{\textnormal{d}}}$. Since $\beta(q)=a_1\cdots a_n^+\overline{a_1\cdots a_n}0^\infty$ we have
\begin{align*}
f_{{\textnormal{c}}, {\textnormal{d}}}(q)&=(1{\textnormal{c}})_q+(m{\textnormal{d}})_q-(m^\infty)_q\\
&=(1\overline{a_1\cdots a_n^+}(a_1\cdots a_n)^\infty)_q+(m0^{2n}(\overline{a_1\cdots a_n})^\infty)_q-(m^\infty)_q\\
&=((m+1)\overline{a_1\cdots a_n^+}a_1\cdots a_n0^\infty)_q-(m^{2n+1}0^\infty)_q\\
&=(10^na_1\cdots a_n0^\infty)_q-(0a_1\cdots a_n^+m^{n}0^\infty)_q\\
&=(10^\infty)_q-(0a_1\cdots a_n^+\overline{a_1\cdots a_n}0^\infty)_q=0.
\end{align*}
(ii) Suppose that $n=1$. Take
\begin{equation*}
{\textnormal{c}}=\overline{a_1^+}a_1^\infty\qtq{and}{\textnormal{d}}=0^{2}(\overline{a_1})^\infty.
\end{equation*}
In this case $\alpha(q)=(a_1^+\overline{a_1^+})^\infty$ and $a_1^+\ge k+1$ ($m=2k$) and $a_1^+\ge k+2$ ($m=2k+1$). Then $\textnormal{c}, \textnormal{d}\in  A'_q$ and $q=q_{{\textnormal{c}},{\textnormal{d}}}$ follow from that facts $\overline{a_1^+}<a_1<a_1^+$ and
\begin{align*}
f_{{\textnormal{c}}, {\textnormal{d}}}(q)&=(1{\textnormal{c}})_q+(m{\textnormal{d}})_q-(m^\infty)_q=(1\overline{a_1^+}a_1^\infty)_q+(m0^{2}(\overline{a_1})^\infty)_q-(m^\infty)_q\\
&=((m+1)\overline{a_1^+}a_10^\infty)_q-(m^{3}0^\infty)_q=(10a_10^\infty)_q-(0a_1^+m0^\infty)_q\\
&=(10^\infty)_q-(0a_1^+\overline{a_1}0^\infty)_q=0.
\end{align*}
So the proof is finished.
\end{proof}
The following result strengthens Theorem \ref{T4}.
\begin{proposition}\label{81}
We write
\begin{align*}
&\mathcal E_2(m):=\set{q\in(1,m+1]:1\in\overline{\uu_q}-\overline{\uu_q}}\\
&\mathcal F_2(m):=\set{q\in(1,m+1]:1\in\vv_q-\vv_q}.
\end{align*}
Then $\bb_2(m)=\mathcal E_2(m)=\mathcal F_2(m)\setminus\set{\g(m)}$.
\end{proposition}
\begin{proof}
By Theorem \ref{T4} we have $\bb_2(m)=\set{q\in(1,m+1]:1\in\uu_q-\uu_q}$. Moreover, $\uu_q\subset\overline{\uu_q}\subset\vv_q$. Hence we have $\bb_2(m)\subset\mathcal E_2(m)\subset\mathcal F_2(m)$.

Applying Lemma \ref{19} (i) we see that $\mathcal E_2(m)\subset\bb_2(m)\cup\overline{\uu}$. Thus it follows from Lemma \ref{12} that $\mathcal E_2(m)=\bb_2(m)$. On the other hand, according to Lemma \ref{19} (ii) we know that $\mathcal F_2(m)\subset\bb_2(m)\cup\vv$. Then it follows from Lemma \ref{13} and $\g(m)\notin\bb_2(m)$ that $\mathcal F_2(m)=\bb_2(m)\cup\set{\g(m)}$.
\end{proof}

\begin{lemma}\label{22}
$
\vv\setminus{\g(m)}\subset\bb^{(2)}_2(m)
$
\end{lemma}
\begin{proof}
It suffices to prove that $\vv\setminus(\overline{\uu}\cup\set{\g(m)})\subset\bb^{(2)}_2(m)$ by Lemma \ref{12}. Fix $q\in\vv\setminus(\overline{\uu}\cup\set{\g(m)})$ arbitrarily. By Lemma \ref{18} (iii) there exists a word $a_1\cdots a_n$ with $n\ge1$ such that $\alpha(q)=(a_1\cdots a_n^+\overline{a_1\cdots a_n^+})^\infty$ and for all $0<i<n$
\begin{equation}\label{75}
\overline{(a_1\cdots a_n)^\infty}\le \sigma^i((a_1\cdots a_n)^\infty)<(a_1\cdots a_n)^\infty.
\end{equation}
Set
\begin{align*}
\textnormal{c}=\overline{a_1\cdots a_n^+}(a_1\cdots a_n)^\infty, \textnormal{d}=0^{2n}(\overline{a_1\cdots a_n})^\infty
\intertext{and}
\textnormal{d}_j=0^{2n}(\overline{a_1\cdots a_n})^j(\overline{a_1\cdots a_n^+}a_1\cdots a_n^+)^\infty.
\end{align*}
(i) Suppose that $n\ge2$. According to the proof of Lemma \ref{13} one gets $\textnormal{c}, \textnormal{d}\in A'_q$ and $q=q_{\textnormal{c},\textnormal{d}}$. We will show that $\textnormal{d}_j\in A'_p$ for all $j\ge1$ and $p\in(q,m+1]$.
Since $q\in\vv$,  by Lemma \ref{17} (iii) it suffices to prove
\begin{align*}
\overline{\alpha(p)}<\sigma^i((\overline{a_1\cdots a_n})^j(\overline{a_1\cdots a_n^+}a_1\cdots a_n^+)^\infty)<\alpha(p)
\end{align*}
hold for all $0\le i<nj$. It follows from \eqref{75} and $\overline{a_1}<a_1$ that
\begin{align*}
\overline{a_1\cdots a_n^+}<\overline{a_1\cdots a_n}\le \overline{a_{i+1}\cdots a_na_1\cdots a_i}\le a_1\cdots a_n<a_1\cdots a_n^+
\end{align*}
for $0\le i<n$. Hence $\textnormal{d}_j\in A'_p$ for all $j\ge1$. Next we prove  $q_{\textnormal{c}, \textnormal{d}_j}\in\bb_2(m)$.
\begin{align*}
f_{\textnormal{c}, \textnormal{d}_j}(q)&=(1\textnormal{c})_q+(m0^{2n}(\overline{a_1\cdots a_n})^j(\overline{a_1\cdots a_n^+}a_1\cdots a_n^+)^\infty)_q-(m^\infty)_q\\
&=(1\textnormal{c})_q+(m0^{2n}(\overline{a_1\cdots a_n})^{j+1}0^\infty)_q-(m^\infty)_q\\
&<(1\textnormal{c})_q+(m0^{2n}(\overline{a_1\cdots a_n})^{\infty})_q-(m^\infty)_q=f_{\textnormal{c}, \textnormal{d}}(q)=0.
\end{align*}
By Corollaries \ref{c1} and \ref{c2} $f_{\textnormal{c}, \textnormal{d}_j}(t)=0$ has a unique root $q_{{\textnormal{c}, \textnormal{d}_j}}$ in $(q,m+1]$ for all $j\ge1$, i.e., $q_{\textnormal{c}, \textnormal{d}_j}\in\bb_2(m)$. Applying Lemma \ref{10} and the continuity of $f_{\textnormal{c}, \textnormal{d}}$ w.r.t.  $(\textnormal{c},\textnormal{d})$, we infer that $q_{{\textnormal{c}, \textnormal{d}_j}}\searrow q$ as $j\rightarrow\infty$.
Let
\begin{equation*}
\textnormal{c}_\ell=\overline{a_1\cdots a_n^+}(a_1\cdots a_n)^\ell(\overline{a_1\cdots a_n^+}a_1\cdots a_n^+)^\infty.
\end{equation*}
 Note that $f_{\textnormal{c}_\ell, \textnormal{d}_j}(q)<f_{\textnormal{c}, \textnormal{d}_j}(q)<0$. By the same argument we  can conclude that for each fixed $j\ge1$ we have $q_{{\textnormal{c}_\ell, \textnormal{d}_j}}\in\bb_2(m)$ for all $\ell\ge1$,  and   $q_{{\textnormal{c}_\ell, \textnormal{d}_j}}\searrow q_{{\textnormal{c}, \textnormal{d}_j}}$ as $\ell\rightarrow\infty$. Then for each $j$, $q_{{\textnormal{c}, \textnormal{d}_j}}\in\bb_2^{(1)}(m)$. 

(ii) Suppose that $n=1$. Then $\alpha(q)=(a_1^+\overline{a_1^+})^\infty$, let
$
\textnormal{c}=\overline{a_1^+}a_1^\infty $, $\textnormal{d}=0^{2}\overline{a_1}^\infty $, $\textnormal{c}_\ell=\overline{a_1^+}(a_1)^\ell(\overline{a_1^+}a_1^+)^\infty$ and
$\textnormal{d}_j=0^{2}(\overline{a_1})^j(\overline{a_1^+}a_1^+)^\infty$.
By the same argument as the first case  we can conclude that  $q\in\bb^{(2)}_2(m)$.
\end{proof}

\begin{proof}[Proof of Theorem \ref{T3}]
We get the result by Theorem \ref{T4} and Proposition \ref{81}.
\end{proof}

\begin{proof}[Proof of Theorem \ref{T1}]
Results (i), (ii) and (iii) follow from Lemmas \ref{21}, \ref{12} and \ref{22}, respectively. Result (iv) follows from Lemma \ref{22} and the fact that the set $(\g(m),q_f(m))\cap\bb_2(m)$ is finite.
\end{proof}

\section{Some results on unique expansions}
Recall that $\uu$ is the set of univoque bases $q$ in which $1$ has unique expansion, $\overline{\uu}$  the topological closure of $\uu$, and $\vv$  the set of bases $q$ in which $1$ has unique doubly infinite expansion. Let
$$
(1,m+1]\setminus\overline{\uu}=\cup(p_0,p^*_0),
$$
where $p_0$ runs over $\set{1}\cup(\overline{\uu}\setminus\uu)$ and $p_0^*$ runs over a proper subset of $\overline{\uu}$. It was proved in \cite{KL} that $p_0$ is an algebraic number while $p^*_0$ is a transcendental number.
Now let
\begin{equation}\label{bigm}
(M,m+1]\setminus\overline{\uu}=\cup(q_0,q^*_0),\; M=\left \lfloor \frac{m}{2} \right \rfloor+1.
\end{equation}
In this section we shall give a description of $\uu '_{q^*_0}$.

 Note from \cite{KL1,VKL}, for each connected component $(q_0,q^*_0)$ there exists a finite word $a_1\cdots a_n$
such that $\alpha(q_0)=(a_1\cdots a_n)^\infty \in \Omega _m ^{\mathbb N}$ where  $a_1\cdots a_n$ is assumed the smallest periodic block. The right endpoint $q^*_0$ is the limit of sequence $\{q_\ell\}$ defined below.
Let
\begin{equation}\label{czero}
c_0^-:=a_1\cdots a_n\;\textrm{and}\; c_{\ell+1}=c_\ell\overline{c_\ell}^+, \;\;\ell=0,1,\cdots.
\end{equation}
Then $(c_i)$ is a Thue-Morse type sequence generated by $c_0^-$ (cf. \cite{A}).
From \cite{VKL} it follows that for each $\ell=1,2,\cdots$ there exists $q_\ell \in (q_0, q_0^*)$ such that $\beta (q_\ell )=c_\ell0^\infty$. de Vries and Komornik
obtained in \cite{KL1,VKL} that for each connected component $(q_0,q^*_0)$
\begin{equation*}
\vv\cap(q_0,q^*_0)=\set{q_\ell;\ell\in\mathbb N}\;\textrm{and}\; q_\ell \uparrow q^*_0.
\end{equation*}

%
%

We recall some standard results established in \cite{KL}:
\begin{lemma}\label{30}
Let $M, c_\ell$ be given in (\ref{bigm}) and (\ref{czero}). Let $(q_0,q^*_0)$ be a connected component of $(M,m+1]\setminus\overline{\uu}$ related to $c_0^-$, and $(d_i)\in\uu_{q^*_0}'$ w.r.t. $\Omega_m$.
\begin{enumerate}[\upshape (i)]
\item If $d_j<m$ and $d_{j+1}\cdots d_{j+2^\ell n}=c_\ell$  for some $\ell\ge0$, then
\begin{equation*}
d_{j+2^\ell n+1}\cdots d_{j+2^{\ell+1}n}=\overline{c_{\ell}}\quad \text{or}
\quad d_{j+2^\ell n+1}\cdots d_{j+2^{\ell+1}n}=\overline{c_{\ell}}^+.
\end{equation*}
\item If $d_j>0$ and $d_{j+1}\cdots d_{j+2^\ell n}=\overline{c_{\ell}}$  for some $\ell\ge0$, then
\begin{equation*}
d_{j+2^\ell n+1}\cdots d_{j+2^{\ell+1}n}=c_{\ell}\quad \text{or}
\quad d_{j+2^\ell n+1}\cdots d_{j+2^{\ell+1}n}=c_{\ell}^-.
\end{equation*}
\end{enumerate}
\end{lemma}

\begin{lemma}\label{29} \cite[Lemma 4.2]{KL}
Let $M, c_\ell$ be given in (\ref{bigm}) and (\ref{czero}).
Let $(q_0,q^*_0)$ be a connected component of $(M,m+1]\setminus\overline{\uu}$ related to $c_0^-$. Then for any $\ell\ge0$, $c_\ell=a_1\cdots a_{2^\ell n}$ satisfies
\begin{equation*}
\overline{a_1\cdots a_{2^\ell n-i}}<a_{i+1}\cdots a_{2^\ell n}\le a_1\cdots a_{2^\ell n-i}
\end{equation*}
for all $0\le i<2^\ell n$.
\end{lemma}
Our first result is

\begin{theorem}\label{61}
Let $M, c_\ell$ be given in (\ref{bigm}) and (\ref{czero}).
Let $(q_0,q^*_0)=(M,q_{KL})$ be the first connected component of $(M,m+1]\setminus\overline{\uu}$ related to $c_0^-$.
\begin{enumerate}[\upshape (I)]


\item Suppose that  $m=2k+1\ge3$. Then $(b_i)\in\uu_{q_{KL}}'\setminus\set{0^\infty,m^\infty}$ if and only if $(b_i)$ is formed by sequences of the form
\begin{equation}\label{251}
\omega (c_0^-)^j(c_{i_1}\overline{c_{i_1}})^{j_1}(c_{i_1}\overline{c_{i_2}})^{l_1}(c_{i_2}\overline{c_{i_2}})^{j_2}(c_{i_2}\overline{c_{i_3}})^{l_2}\cdots
\end{equation}
or their reflections, where
$0\le i_1<i_2<\cdots $ are integers, $1\leq j\leq 2$, $0\leq l_i\leq 1$,  $0\le  j_i\le\infty $ for all $i\ge 1$ and
\begin{align*}
 \omega \in & \{1,\cdots, m-1\}\cup \bigcup _{N=1}^\infty \{0^Nb: 0<b\leq k+1\}\\
 &\cup \bigcup _{N=1}^\infty \{m^Nb: k\leq b <m\}.
\end{align*}

\item Suppose that  $m=2k$. Then $(b_i)\in\uu_{q_{KL}}'\setminus\set{0^\infty,m^\infty}$ if and only if $(b_i)$ is formed by sequences of the form
\begin{equation}\label{25}
\omega (c_0^-)^j(c_{i_1}\overline{c_{i_1}})^{j_1}(c_{i_1}\overline{c_{i_2}})^{l_1}(c_{i_2}\overline{c_{i_2}})^{j_2}(c_{i_2}\overline{c_{i_3}})^{l_2}\cdots
\end{equation}
or their reflections, where  $0\le i_1<i_2<\cdots $ are integers, $0\leq l_i\leq 1$,  $0\le j, j_i\le\infty $ for all $i\ge 1$ and
\begin{align*}
\omega \in & \{1,\cdots, m-1\}\cup \bigcup _{N=1}^\infty \{0^Nb: 0<b\leq k+1\}\\
&\cup \bigcup _{N=1}^\infty \{m^Nb: k-1\leq b <m\}.
\end{align*}

\end{enumerate}
\end{theorem}
\begin{remark}
The case $m=1$ was studied in \cite{KK}, which is quite different from the case $m=2k+1\ge3$. So we assume $m=2k+1\ge3$ in Theorem \ref{61} (I).
\end{remark}
\begin{proof}
Since $(q_0,q^*_0)=(M,q_{KL})$, thus $\uu_{q_0}'=\set{0^\infty,m^\infty}$. Note that whether $m=2k$ or $m=2k+1$ we always have $q_0=M=k+1$ and so $\alpha (q_0)=k^\infty $, $c_0^-=k$.

For the sufficiency, it is not difficult to be verified by Lemma \ref{29}. We leave it for the readers.

In the following, we prove the necessity.  Take $(b_i)\in\uu_{q_0^*}'\setminus\uu_{q_0}'$. Let
$$
N=\min \{s: 0<b_s<m\}.
$$
Then $N$ is well-defined and is  a positive integer.

{\bf CASE I}. $m=2k+1$ with $k\ge1$.

 Note from \eqref{KL} $\alpha (q_{KL})=(k+1)(k+1)k(k+1)\cdots $.  By (\ref{uicodedes1}) we have for $t\ge 1$
\begin{equation}\label{li1031}
\begin{split}
&b_{t+1}b_{t+2}\cdots<(k+1)(k+1)k(k+1)\cdots ,\textrm{when}\; b_1\cdots b_t\ne m^t \\
&b_{t+1}b_{t+2}\cdots>kk(k+1)k\cdots ,  \textrm{when}\; b_1\cdots b_t\ne 0^t.
\end{split}
\end{equation}
We now split our discussion into two steps.

{\bf Step 1}. We shall show what does the block $\omega=b_1\cdots b_N$ look like. The case $N=1$ is trivial, we only consider the case $N>1$.

Subcase 1. Suppose  that  $\omega$ begins at $0$. Then by (\ref{li1031}) we have $\omega=0^{N-1}b_N$ with $0<b_{N}\leq k+1$.

Subcase 2.  Suppose that $\omega$ begins at $m$. Then by (\ref{li1031}) we have $\omega=m^{N-1}b_N$ with $k\leq b_{N}<m$.

{\bf  Step 2}. Now let us to explore the sequence $(b_{N+i})=(b_{N+i})_{i\geq 1}$.
Note  that $0<b_N<m$ and $\alpha (q_{KL})=(k+1)(k+1)k(k+1)\cdots $. Thus, by Lemma \ref{L22} we have for each $i\ge 1$
\begin{equation}\label{oct9}
kk(k+1)k\cdots<b_{N +i}b_{N +i+1}\cdots <(k+1)(k+1)k(k+1)\cdots ,
\end{equation}
and so
\begin{equation*}
k\leq b_{N+i}\leq k+1,\; (b_{N+i})_{i\geq 1}\;\textrm{not ending with}\; k^\infty \;\textrm{or}\; (k+1)^\infty .
\end{equation*}
From (\ref{oct9}) it follows that there exists $1\leq j\leq 2$ such that either $b_{N+1}\cdots b_{N+j}=k^j$ or $b_{N+1}\cdots b_{N+j}=(k+1)^j$. Without loss of generality, we assume that $b_{N+1}\cdots b_{N+j}=k^j$. Otherwise, we only need to consider
 $(\overline{b_{i}})_{i\geq 1}$ instead. So $b_{N+j+1}=k+1=c_0$. In the following we shall determine the tail $(b_{N+j+i})_{i>1}$ by means of Lemma \ref{30}.

Let us recall that
\begin{align}\label{e57}
\begin{split}
&c_0=\overline{c_0}^+=k+1,\; \;c_0^-=\overline{c_0}=k\\
&c_{\ell+1}=c_\ell\overline{c_\ell}^+,\;\; \overline{c_{\ell+1}}=\overline{c_\ell}c_\ell^-.\\
\end{split}
\end{align}
Roughly speaking, Lemma \ref{30} tells us that which possible blocks will follow a block $c_\ell$ or $\overline{c_\ell }$. This can be simply described in Figure 1.
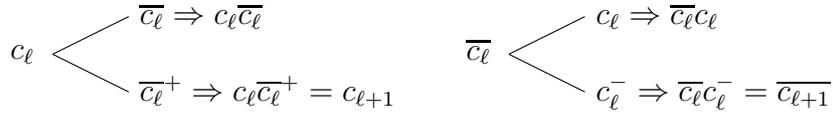
\begin{figure}[h]\label{figure1}
\centering
\begin{tikzpicture}[scale=5]
\draw(0,.3)--(.2,.4)node[right]{\normalsize $\overline{c_\ell}\Rightarrow c_{\ell }\overline{c_{\ell }}$};
\draw(0,.3)--(.2,.2)node[right]{\normalsize $\overline{c_\ell}^+\Rightarrow c_{\ell }\overline{c_{\ell }}^+=c_{\ell +1}$};

\draw(1.2,.3)--(1.4,.4)node[right]{\normalsize $c_\ell\Rightarrow \overline{c_{\ell }}c_{\ell }$};
\draw(1.2,.3)--(1.4,.2)node[right]{\normalsize $c_\ell^-\Rightarrow \overline{c_{\ell }}c_{\ell }^-=\overline{c_{\ell +1}}$};



\node [label={[xshift=-.4cm, yshift=1.08cm]\normalsize $c_\ell$}] {};

\node [label={[xshift=5.6cm, yshift=1.08cm]\normalsize $\overline{c_\ell}$}] {};
\end{tikzpicture}\caption{Relation induced by Lemma \ref{30}}
\end{figure}
By $A\rightharpoonup B$ we denote block $A$ followed by block $B$. We point out that in figure 1 the action $c_\ell \rightharpoonup \overline{c_\ell }^+$ cannot be implemented continuously infinite times since
 $(b_i)\in\uu_{q_{KL}}'$ cannot be ended with $\alpha (q_{KL})$. Similarly, the action $  \overline{c_\ell }\rightharpoonup c_\ell^-$ cannot be implemented  continuously infinite times.
%
%

Now we have $b_{N+j+1}=k+1=c_0$ and $b_{N+j}=k<m$. Then following block is either $\overline{c_0}$ or $\overline{c_0}^+$ by Lemma \ref{30}, i.e.,
\begin{equation}\label{oct10yi}
(b_i)=\omega (c_0^-)^jc_0\overline{c_0}(b_{N+j+2+i})_{i\geq 1}
\end{equation}
or
\begin{equation}\label{oct10er}
(b_i)=\omega (c_0^-)^jc_0\overline{c_0}^+(b_{N+j+2+i})_{i\geq 1}=\omega (c_0^-)^jc_1(b_{N+j+2+i})_{i\geq 1}
\end{equation}
by (\ref{e57}).
If (\ref{oct10yi}) occurs, then
$$
(b_i)=\omega (c_0^-)^jc_0\overline{c_0}c_0\cdots \;\;\textrm{or}\;\; (b_i)=\omega (c_0^-)^jc_0\overline{c_0}c_0^-\cdots=\omega (c_0^-)^jc_0\overline{c_1}\cdots
$$
by Lemma \ref{30} and (\ref{e57}).
 If (\ref{oct10er}) occurs, then
$$
(b_i)=\omega (c_0^-)^jc_1\overline{c_1}\cdots \;\;\textrm{or}\;\;
(b_i)=\omega (c_0^-)^jc_1\overline{c_1}^+\cdots=\omega (c_0^-)^jc_2\cdots
$$
by Lemma \ref{30} and (\ref{e57}).

In any cases described above, one can continuous to implement the  process in the same way as above. Therefore, $(b_i)$ is of form (\ref{251}) or its reflection.

{\bf CASE II}. $m=2k$ with $k\ge1$.

 Note that $\alpha(q_{KL})=(k+1)k(k-1)(k+1)\cdots$.  By (\ref{uicodedes1}) we have for $t\ge 1$
\begin{equation}\label{wx1031}
\begin{split}
&b_{t+1}b_{t+2}\cdots<(k+1)k(k-1)(k+1)\cdots ,\textrm{when}\; b_1\cdots b_t\ne m^t \\
&b_{t+1}b_{t+2}\cdots>(k-1)k(k+1)(k-1)\cdots,  \textrm{when}\; b_1\cdots b_t\ne 0^t.
\end{split}
\end{equation}
 We now split our discussion into two steps.

{\bf Step 1}. We shall show what does the block $\omega=b_1\cdots b_N$ looks like. As Case I we only consider $N>1$.

Case 1. Suppose  that  $\omega$ begins at $0$. Then by (\ref{wx1031}) we have $\omega=0^{N-1}b_N$ with $0<b_{N}\leq k+1$.

Case 2.  Suppose that $\omega$ begins at $m$. Then by (\ref{wx1031}) we have $\omega=m^{N-1}b_N$ with $k-1\leq b_{N}<m$.

{\bf  Step  2}. Now let us to explore the sequence $(b_{N+i})=(b_{N+i})_{i\geq 1}$.
Note  that $0<b_N<m$ and $\alpha (q_{KL})=(k+1)k(k-1)(k+1)\cdots$. Thus, by Lemma \ref{L22} we have for each $i\geq 1$
\begin{equation}\label{nov9}
(k-1)k(k+1)(k-1)\cdots<b_{N +i}b_{N +i+1}\cdots <(k+1)k(k-1)(k+1)\cdots.
\end{equation}
From (\ref{nov9}) it follows that
$$
b_{N+1}\in \{k-1, k, k+1\}=\{\overline{c_0}, c_0^-, c_0\}.
$$

 {\bf I}. $b_{N+1}=k$.

In this case,  either $(b_{N+i})_{i\geq 1}=k^\infty =(c_0^-)^\infty $ or there exists a positive integer $j$ such that $b_{N+1}\cdots b_{N+j+1}\in \{k^j(k+1), k^j(k-1)\}=\{(c_0^-)^jc_0, (c_0^-)^j\overline{c_0}\}$.

Without loss of generality, we assume that $b_{N+1}\cdots b_{N+j+1}=k^j(k+1)=(c_0^-)^jc_0$. Otherwise, we only need to consider
 $(\overline{b_{i}})_{i\geq 1}$ instead. So $b_{N+j+1}=k+1=c_0$.

  {\bf II}. $b_{N+1}\in \{k-1, k+1\}=\{\overline{c_0}, c_0\}$.

 Without loss of generality, we assume that $b_{N+1}=k+1=c_0$. Otherwise, we only need to consider
 $(\overline{b_{i}})_{i\geq 1}$ instead. For the sake of uniformity, we write  $b_{N+1}=k+1=(c_0^-)^0c_0$, corresponding to $j=0$.

 In the following we shall determine the tail $(b_{N+j+1+i})_{i\geq 1}$ by means of Lemma \ref{30}, where $j$ is a nonnegative integer.

Let us recall that
\begin{align}\label{e53}
\begin{split}
&c_0^-=\overline{c_0}^+=k,\; \;c_0=k+1,\; \overline{c_0}=k-1\\
&c_{\ell+1}=c_\ell\overline{c_\ell}^+,\;\; \overline{c_{\ell+1}}=\overline{c_\ell}c_\ell^-.\\
\end{split}
\end{align}
Roughly speaking, Lemma \ref{30} tells us that which block can follow a block $c_\ell$ or $\overline{c_\ell }$. This can be simply described in Figure 1.





 We point out that in figure 1 the action $c_\ell \rightharpoonup \overline{c_\ell }^+$ cannot be implemented continuously infinite times since
 $(b_i)\in\uu_{q_{KL}}'$ cannot be ended with $\alpha (q_{KL})$. Similarly, the action $  \overline{c_\ell }\rightharpoonup c_\ell^-$ cannot be implemented  continuously infinite times.
%
%

Now we have $b_{N+j+1}=k+1=c_0$ and $b_{N+j}<m$. Then following block is either $\overline{c_0}$ or $\overline{c_0}^+$ by Lemma \ref{30}, i.e.,
\begin{equation}\label{nov10yi}
(b_i)=\omega (c_0^-)^jc_0\overline{c_0}(b_{N+j+2+i})_{i\geq 1}
\end{equation}
or
\begin{equation}\label{nov10er}
(b_i)=\omega (c_0^-)^jc_0\overline{c_0}^+(b_{N+j+2+i})_{i\geq 1}=\omega (c_0^-)^jc_1(b_{N+j+2+i})_{i\geq 1}
\end{equation}
by (\ref{e53}).
If (\ref{nov10yi}) occurs, then
$$
(b_i)=\omega (c_0^-)^jc_0\overline{c_0}c_0\cdots \;\;\textrm{or}\;\; (b_i)=\omega (c_0^-)^jc_0\overline{c_0}c_0^-\cdots=\omega (c_0^-)^jc_0\overline{c_1}\cdots
$$
by Lemma \ref{30} and (\ref{e53}).
 If (\ref{nov10er}) occurs, then
$$
(b_i)=\omega (c_0^-)^jc_1\overline{c_1}\cdots \;\;\textrm{or}\;\;
(b_i)=\omega (c_0^-)^jc_1\overline{c_1}^+\cdots=\omega (c_0^-)^jc_2\cdots
$$
by Lemma \ref{30} and (\ref{e53}).

In any cases described above, one can continuous to implement the  process in the same way as above. Therefore, $(b_i)$ is of form (\ref{25}) or its reflection.
\end{proof}


Let $(q_0,q^*_0)$  be a connected component of $(M,m+1]\setminus\overline{\uu}$. Suppose that $\alpha (q_0)=(a_1\cdots a_n)^\infty =(c_0^-)^\infty $.
 For a word $v_1v_2\cdots v_p\in \{0,1,\cdots , m\}^p$ and $1\leq q\leq p$, let
 \begin{equation*}
(v_1\cdots v_p)|_q=v_1\cdots v_q \qtq{and} \sigma (v_1v_2\cdots v_p)=v_2\cdots v_p.
 \end{equation*}
 A word $u_1\cdots u_p\in \{0,1,\cdots , m\}^p$
is \emph{matched} to $a_1\cdots a_n$ if $u_p<m$ and for any $1\leq \ell \leq p$
$$
\sigma ^\ell (u_1\cdots u_pa_1\cdots a_{n})|_n \leq a_1\cdots a_n\;\;\textrm{whenever}\;\; u_1\cdots u_\ell \ne m^\ell
$$
and
$$
\sigma ^\ell (u_1\cdots u_pa_1\cdots a_{n})|_n \geq \overline{a_1\cdots a_n}\;\;\textrm{whenever}\;\; u_1\cdots u_\ell \ne 0^\ell.
$$

Obviously, if $u_1\cdots u_p$ is matched to $a_1\cdots a_n$, then
\begin{equation*}
\overline{u_1\cdots u_p}\qtq{and} u_1\cdots u_pa_1\cdots a_n
\end{equation*}
are also matched to $a_1\cdots a_n$.
\begin{theorem}\label{62}
Let $M, c_\ell$ be given in (\ref{bigm}) and (\ref{czero}).
If $(q_0,q^*_0)$ is a connected component of $(M,m+1]\setminus\overline{\uu}$ related to $c_0^-=a_1\cdots a_n$ with $q_0>M$. Then $(b_i)\in\uu_{q_0^*}'\setminus\uu_{q_0}'$ if and only if $(b_i)$ is formed by the sequences of form
\begin{equation*}\label{e516}
\omega (c_0^-)^j(c_{i_1}\overline{c_{i_1}})^{j_1}(c_{i_1}\overline{c_{i_2}})^{l_1}(c_{i_2}\overline{c_{i_2}})^{j_2}(c_{i_2}\overline{c_{i_3}})^{l_2}\cdots
\end{equation*}
or their reflections, where
\begin{equation*}
\omega \in \bigcup _{p=1}^\infty \{u_1\cdots u_p: u_1\cdots u_p\; \textrm{is matched to}\; a_1\cdots a_n\},
\end{equation*}
$0\le i_1<i_2<\cdots $ are integers, $0\leq l_i\leq 1$,  $j\in \{0, \infty \}$ and $0\le  j_i\le\infty $ for all $i\ge 1$.
\end{theorem}
\begin{proof}

As described in (\ref{czero}), we have
$$
\alpha (q_0)=(a_1\cdots a_n)^\infty =(c_0^-)^\infty
$$
and
$$
\alpha (q_0^*)=c_0\overline{c_0}^+\overline {c_0}c_0\cdots =a_1\cdots a_{n-1}a_n^+\overline{a_1\cdots a_n}\cdots
$$
Now taking a $(b_i)\in\uu_{q_0^*}'\setminus\uu_{q_0}'$.

{\bf CASE I}. $(b_i)$  ends with neither $(c_0^-)^\infty$ nor  $\overline{(c_0^-)^\infty}$.

Since $(b_i)\notin \uu_{q_0}'$, there exist a smallest positive integer $\eta $ such that
\begin{equation}\label{1105li}
b_1\cdots b_\eta \ne m^\eta \;\;\textrm{and}\;\; b_{\eta +1}\cdots b_{\eta +n}> c_0^-=a_1\cdots a_n,
\end{equation}
or
$$
b_1\cdots b_\eta \ne 0^\eta \;\;\textrm{and}\;\; b_{\eta +1}\cdots b_{\eta +n}< \overline{c_0^-}=\overline{a_1\cdots a_n}.
$$
Without loss of generality we assume that (\ref{1105li}) holds. Otherwise one can consider $\overline{(b_i)}$ instead.

On the other hand, by Lemma \ref{L22} we have
$$
b_{\eta +1}\cdots b_{\eta +n}\cdots <c_0\overline{c_0}^+\overline {c_0}c_0\cdots
$$
Thus, combining (\ref{1105li}) one can get
\begin{equation*}
b_1\cdots b_\eta \ne m^\eta \;\;\textrm{and}\;\; b_{\eta +1}\cdots b_{\eta +n}=c_0=a_1\cdots a_n^+.
\end{equation*}
We claim that $b_\eta <m$. This is clear when $\eta =1$.
Suppose that  $\eta >1$ and $b_\eta =m$. Then by the minimality we have
$$
b_1\cdots b_{\eta -1} \ne m^{\eta -1} \;\;\textrm{and}\;\; b_\eta b_{\eta +1}\cdots b_{\eta +n-1}\leq c_0^-=a_1\cdots a_n.
$$
By (\ref{1105li}) this implies that $a_1\cdots a_n=m^n$ and so $q_0=m+1$, a contradiction. Hence we get
\begin{equation*}
 b_\eta < m \;\;\textrm{and}\;\; b_{\eta +1}\cdots b_{\eta +n}=c_0=a_1\cdots a_n^+.
\end{equation*}
By the same argument as that we did in Theorem \ref{61}, one can get $(b_{\eta +i})_{i\geq 1}$ is of form
$$
(c_{i_1}\overline{c_{i_1}})^{j_1}(c_{i_1}\overline{c_{i_2}})^{l_1}(c_{i_2}\overline{c_{i_2}})^{j_2}(c_{i_2}\overline{c_{i_3}})^{l_2}\cdots .
$$

Next let us investigate what the prefix $b_1\cdots b_\eta $ looks like.


By the definition of $\eta $ we can denote
$$
b_1\cdots b_\eta =m^ub_{u+1}\cdots b_{\eta },
$$
where $u<\eta $ is a nonnegative integer, $b_{u+1}, b_{\eta }<m$. Then by (\ref{1105li})

\begin{equation*}
\begin{split}
&\sigma ^{p}(b_1\cdots b_\eta a_1\cdots a_n)|_n=\sigma ^p(b_1\cdots b_\eta b_{\eta +1}\cdots b_{\eta +n-1}b_{\eta +n}^- )|_n\\
&\leq a_1\cdots a_n \;\;\textrm{for}\;\; u<p\leq \eta
\end{split}
\end{equation*}
and
\begin{equation*}
\begin{split}
&\sigma ^{p}(b_1\cdots b_\eta a_1\cdots a_n)|_n=\sigma ^p(b_1\cdots b_\eta b_{\eta +1}\cdots b_{\eta +n-1}b_{\eta +n}^- )|_n\\
&\geq \overline{ a_1\cdots a_n} \;\;\textrm{for}\;\; 1\leq p\leq \eta \;\textrm{with}\;
b_1\cdots b_p\ne 0^p.
\end{split}
\end{equation*}
Therefore, $b_1\cdots b_\eta $ is matched to $a_1\cdots a_n$.

{\bf CASE II}. $(b_i)$  ends with either $(c_0^-)^\infty$ or  $\overline{(c_0^-)^\infty}$.

Without loss of generality, we assume  $(b_i)$  ends with  $(c_0^-)^\infty$. Otherwise one only needs to consider $\overline{(b_i)}$ instead.  Then $(b_i)$ can be written as
$$
(b_i)=b_1\cdots b_p (a_1\cdots a_n)^\infty \;\text{where $p$ is a positive integer}.
$$
We claim that $b_1\cdots b_p $ is matched to $a_1\cdots a_n$. Otherwise there exists a smallest positive integer $1\leq \ell  \leq p$ such that
\begin{equation*}
\begin{split}
&\sigma ^\ell (b_1\cdots b_pa_1\cdots a_{n})|_n\\
=&\sigma ^\ell (b_1\cdots b_pb_{p+1}\cdots b_{p+n})|_n> a_1\cdots a_n\;\;\textrm{and}\;\; b_1\cdots b_\ell \ne m^\ell
\end{split}
\end{equation*}
or
\begin{equation*}
\begin{split}
&\sigma ^\ell (b_1\cdots b_pa_1\cdots a_{n})|_n\\
=&\sigma ^\ell (b_1\cdots b_pb_{p+1}\cdots b_{p+n})|_n< \overline{a_1\cdots a_n}\;\;\textrm{and}\;\; b_1\cdots b_\ell \ne 0^\ell.
\end{split}
\end{equation*}
Then by the same argument as that in case  I, we have that
$$
b_\ell <m \;\textrm{and} \; b_{\ell +1}\cdots b_{\ell +n}=c_0=a_1\cdots a_{n-1}a_n^+
$$
or
$$
b_\ell >0 \;\textrm{and} \; b_{\ell +1}\cdots b_{\ell +n}=\overline{c_0}=\overline{a_1\cdots a_{n-1}a_n^+}.
$$
So $(b_i)$ cannot end with $(c_0^-)^\infty$ by the argument in case I, which contradicts our hypothesis.

The sufficiency can be checked directly.
\end{proof}

Recall that for each connected component $(q_0,q^*_0)$, we have $(q_0,q^*_0)\cap\vv=\set{q_\ell:\ell\ge 1}$. If $q\in(q_\ell,q_{\ell+1}]$, then $\alpha(q)\le\alpha(q_{\ell+1})=(c_\ell\overline{c_\ell})^\infty$ by Lemma \ref{16}, it follows from Lemma \ref{30} that the blocks $w c_\ell$ and $\overline{w c_\ell}$ are forbidden in each $(b_i)\in\uu_q'$, where $w\in\set{0,\cdots,m-1}$. Otherwise, $(b_i)$ would end with $(c_\ell\overline{c_\ell})^\infty$, which leads to contradiction by Lemma \ref{L22}. So by Theorems \ref{61} and \ref{62} we obtain the following results.

\begin{corollary}\label{63}
Let $M, c_\ell$ be given in (\ref{bigm}) and (\ref{czero}).
Let $(q_0,q^*_0)=(M,q_{KL})$ be the first connected component of $(M,m+1]\setminus\overline{\uu}$ related to $c_0^-$.
\begin{enumerate}[\upshape (I)]
\item Suppose that $q\in(q_\ell,q_{\ell+1}]$ for some $\ell\ge 1$ and $m=2k+1\ge3$. Then $(b_i)\in\uu_q'\setminus\set{0^\infty,m^\infty}$ if and only if $(b_i)$ is formed by sequences of the form
\begin{equation*}
\omega (c_0^-)^j(c_{i_1}\overline{c_{i_1}})^{j_1}(c_{i_1}\overline{c_{i_2}})^{l_1}(c_{i_2}\overline{c_{i_2}})^{j_2}\cdots(c_{i_{n-1}}\overline{c_{i_n}})^{l_{n-1}}(c_{i_n}\overline{c_{i_n}})^{j_n}
\end{equation*}
or their reflections, where
$0\le i_1<i_2<\cdots<i_n<\ell$ are integers, $1\leq j\leq 2$, $0\leq l_i\leq 1$,  $0\le  j_i\le\infty $ for all $i\ge 1$ and
\begin{align*}
 \omega \in & \{1,\cdots, m-1\}\cup \bigcup _{N=1}^\infty \{0^Nb: 0<b\leq k+1\}\\
 &\cup \bigcup _{N=1}^\infty \{m^Nb: k\leq b <m\}.
\end{align*}

\item Suppose that $q\in(q_\ell,q_{\ell+1}]$  for some $\ell\ge 1$ and $m=2k$. Then $(b_i)\in\uu_q'\setminus\set{0^\infty,m^\infty}$ if and only if $(b_i)$ is formed by sequences of the form
\begin{equation*}
\omega (c_0^-)^j(c_{i_1}\overline{c_{i_1}})^{j_1}(c_{i_1}\overline{c_{i_2}})^{l_1}(c_{i_2}\overline{c_{i_2}})^{j_2}\cdots(c_{i_{n-1}}\overline{c_{i_n}})^{l_{n-1}}(c_{i_n}\overline{c_{i_n}})^{j_n}
\end{equation*}
or their reflections, where  $0\le i_1<i_2<\cdots<i_n<\ell $ are integers, $0\leq l_i\leq 1$,  $0\le j, j_i\le\infty $ for all $i\ge 1$ and
\begin{align*}
\omega \in & \{1,\cdots, m-1\}\cup \bigcup _{N=1}^\infty \{0^Nb: 0<b\leq k+1\}\\
&\cup \bigcup _{N=1}^\infty \{m^Nb: k-1\leq b <m\}.
\end{align*}
\end{enumerate}
\end{corollary}

\begin{corollary}\label{64}
Let $M, c_\ell$ be given in (\ref{bigm}) and (\ref{czero}).
If $(q_0,q^*_0)$ is a connected component of $(M,m+1]\setminus\overline{\uu}$ related to $c_0^-=a_1\cdots a_n$ with $q_0>M$, and $q\in(q_\ell,q_{\ell+1}]$ for some $\ell\ge 1$. Then $(b_i)\in\uu_q'\setminus\uu_{q_0}'$ if and only if $(b_i)$ is formed by sequences of the form
\begin{equation*}
\omega (c_0^-)^j(c_{i_1}\overline{c_{i_1}})^{j_1}(c_{i_1}\overline{c_{i_2}})^{l_1}(c_{i_2}\overline{c_{i_2}})^{j_2}\cdots(c_{i_{n-1}}\overline{c_{i_n}})^{l_{n-1}}(c_{i_n}\overline{c_{i_n}})^{j_n}
\end{equation*}
or their reflections, where
\begin{equation*}
\omega \in \bigcup _{p=1}^\infty \{u_1\cdots u_p: u_1\cdots u_p\; \textrm{is matched to}\; a_1\cdots a_n\},
\end{equation*}
$0\le i_1<i_2<\cdots<i_n<\ell $ are integers, $0\leq l_i\leq 1$,  $j\in \{0, \infty \}$ and $0\le  j_i\le\infty $ for all $i\ge 1$.
\end{corollary}

\section*{Acknowlegements}
The authors would like to thank Vilmos Komornik for some suggestions on the
original versions of the manuscript, and for his generous hospitality during the first author's visit in the Department of Mathematics of the University
of Strasbourg. The authors are supported by NSFC No. 12071148, 11971079 and Science and Technology Commission of Shanghai Municipality (STCSM) No.~18dz2271000.

\end{document}